\documentclass[11pt, psamsfonts]{amsart}
\pdfoutput=1
\usepackage{mathpazo, eulervm}
\usepackage[english]{babel}
\usepackage[utf8]{inputenc}
\usepackage{graphicx, amssymb, amsmath}

\newcommand{\X}{\mathbb{X}}
\newcommand{\Y}{\mathbb{Y}}
\newcommand{\R}{\mathbb{R}}
\newcommand{\Z}{\mathbb{Z}}
\newcommand{\N}{\mathbb{N}}
\newcommand{\LL}{\mathcal{L}}

\newcommand{\D}{\mathcal{D}}
\newcommand{\E}{\mathbb{E}}
\renewcommand{\P}{\mathbb{P}}
\newcommand{\AC}{\mathcal{AC}}
\renewcommand{\S}{\mathbb{S}}
\newcommand{\dd}{\mathrm{d}}
\newcommand{\e}{\varepsilon}
\newcommand{\vs}{\mathbf{V}}
\newcommand{\vb}{\mathbf{v}}
\newcommand{\wb}{\mathbf{w}}
\newcommand{\im}{\mathrm{Im\,}}
\newcommand{\homt}{\mathfrak{g}}
\newcommand{\W}{\mathfrak{W}}
\newcommand{\w}{\mathfrak{w}}
\newcommand{\reach}{\mathrm{reach}}
\newcommand{\cl}[1]{\overline{#1}}

\newcommand{\ignore}[1]{}

\newtheorem{lem}{Lemma}
\newtheorem{thm}{Theorem}
\newtheorem{prp}{Proposition}

\newtheorem{cor}{Corollary}
\theoremstyle{remark}

\newcommand{\Hawaii}{Hawai\kern.05em`\kern.05em\relax i }
\newcommand{\Manoa}{M\=anoa }

\title{Typical representatives of free homotopy classes in a
multi-punctured plane}

\author[M. Arnold]{Maxim Arnold}
\address{Department of Mathematical Sciences, University of Texas at
Dallas, Richardson TX, USA}
\email{maxim.arnold@utdallas.edu}
\author[Y. Baryshnikov]{Yuliy Baryshnikov}
\address{Departments  of Mathematics and Electrical and Computer Engineering, University  of Illinois at Urbana-Champaign, Urbana IL,
USA}
\email{ymb@illinois.edu}
\author[Y. Mileyko]{Yuriy Mileyko}
\address{Department of Mathematics, University of \Hawaii at \Manoa,
Honolulu HI, USA}
\email[Corresponding author]{yury@math.hawaii.edu}

\date{\today}

\begin{document}

\begin{abstract}
We show that a uniform probability measure supported on a specific set
of piecewise linear loops in a non-trivial free homotopy class in a
multi-punctured plane is overwhelmingly concentrated around loops of
minimal lengths. Our approach is based on extending Mogulskii's theorem
to closed paths, which is a useful result of independent interest. In
addition, we show that the above measure can be sampled using standard
Markov Chain Monte Carlo techniques, thus providing a simple methods for
approximating shortest loops.
\end{abstract}

\maketitle

\section{Introduction}
\label{sec:intro}
The problem of finding a path of minimum length in a metric space under
topological constraints is one of the classical problems in geometric
optimization. It has numerous applications, including path planning and
navigation \cite{bhattacharya2012, yershov2013}, VLSI routing \cite{gerez1999,
sherwani2012}, and surface cutting \cite{erickson2004}, which is an important
step in surface parametrization \cite{floater1997, sheffer2000} and texture
mapping \cite{bennis1991, piponi2000}.

The shortest path problem has been considered in many different settings, and
tackled using a variety of techniques. Most commonly, paths in a planar domain
or in a (two-dimensinal) surface are considered, and numerous algorithms have
been developed to find the corresponding shortest paths or approximations
thereof (see e.g.  \cite{hershberger1994, grigoriev1998, bespamyatnikh2003,
efrat2006, cabello2010, cheng2012} and references therein).

In this paper we take a completely different approach to this classical
problem.  It has been noted that values of cost functions in some optimization
problems differ very slightly from the mean (or median) value with respect to
some naturally defined probability measure, leading to interesting
approximation techniques \cite{barvinok1997}. This is a consequence of the well
studied concentration of measure phenomenon \cite{ledoux2005}. Roughly
speaking, a Borel probability measure $\mu$ on a metric space $(\X,\dd)$ is
concentrated around a set $A\subset\X$ if the quantity $1-\mu(A_{\e})$, where
$A_{\e}=\{x\in\X|\dd(x,A)<\e\}$, decreases very fast (e.g. exponentially) as
$\e$ grows. A typical example, mentioned in the above references, is the
concentration of the uniform probability measure on a high-dimensional unit
sphere around every equator.

Clearly, an approximate solution to an optimization problem may be obtained by
sampling from a measure concentrated around the minimizers of the cost
function. Of course, constructing such a probability measure, or showing that a
particular measure has the right concentration property, is by no means a
trivial task.  The goal of this paper is to show that such an approach is
indeed viable for the problem of finding loops of minimal length in a fixed,
nontrivial free homotopy class (we define the relevant notions below). 

Specifically, we consider discretized loops in a
multi-punctured plane and show that the uniform probability measure supported
on a specific set of such piecewise linear loops in a non-trivial homotopy class is
overwhelmingly concentrated around loops of minimal lengths. The choice of a
multi-punctured plane provides a nice compromise between simplicity and
applicability, as it can serve as a model for domains in many path planning
applications. We should also mention that our approach is based on extending
the Mogulskii's theorem to closed paths (in the plane), which is a useful
result of independent interest. 

The rest of the paper is structured as follows. Section \ref{sec:prelim}
contains the necessary background information. The statements of our main
results are provided in Section \ref{sec:loop_ldp}. In Section
\ref{sec:sampling} we show that the measure under consideration can be sampled
using standard Markov Chain Monte Carlo techniques. All the proofs of our
results have been put in a separate Section \ref{sec:proofs}. Section
\ref{sec:conclusion} concludes the paper.
\section{Preliminaries}
\label{sec:prelim}
Before stating our main result we need to introduce the necessary
nomenclature and provide several auxiliary results. Additional background
information can be found in such comprehensive texts as
\cite{burago2001, bridson1999, dembo2009}.

\subsection{Paths and loops}
Let $(\X, \dd)$ be a metric space.  A \emph{path} in $\X$ is a continuous
map $\gamma:I\to\X, I:=[0,1]$.  If a path $\gamma$ is closed, that is,
$\gamma(0)=\gamma(1)$, then we call it a \emph{loop}. A loop in $\X$ may
also be regarded as a continuous map from a circle, $\gamma:\S^1\to\X$,
in which case it is convenient to think of the circle as a quotient of
$\R$, $\S^1=\R/\Z$, and regard $\R$ as a covering space for $\S^1$. Such
a setting allows us to consider the lift of a map on $\S^1$ to a map on
$\R$, which is often useful (see e.g. \cite{burago2001} for details).
Given $[a,b]\subset [0,1]$, the restriction of a path $\gamma$ onto
$[a,b]$, denoted by $\gamma|_{[a,b]}$, is the path defined by
$\gamma|_{[a,b]}(t) = \gamma\left(a+{t}{(b-a)}\right)$.  


If paths $\gamma_1,\ldots,\gamma_m$ are such that
$\gamma_i(1)=\gamma_{i+1}(0)$, $i=1,\ldots,m-1$, and
$\mathbf{c}=(c_1,\ldots,c_m)$ is such that $c_i\geq 0$ and 
$\sum_{i=1}^{m}{c_i}=1$, then we define the
$\mathbf{c}$-concatenation of $\gamma_i$
as the path
$$
\gamma_1\overset{c_1}{\cdot}\ldots\overset{c_{m-1}}{\cdot}\gamma_m(t)=
\gamma_i\left(\frac{t-C_{i-1}}{c_i}\right),\quad
t\in[C_{i-1},C_i],
$$
where $C_i=\sum_{j=1}^{i}{c_j}$. The value $c_i$ is called
the \emph{traversal time} of the path $\gamma_i$ in the concatenation.
Note that zero traversal times are allowed only for constant paths,
i.e. paths $\gamma$ such that $\gamma(t)=\gamma(0)$ for all
$t\in[0,1]$.
If traversal times are not important, we will talk about
\emph{a concatenation} of paths.
In this case we will use notation $\gamma_1\cdot\ldots\cdot\gamma_m$.

The length of a path $\gamma$ is defined by
$$
L(\gamma) = \sup{\sum_{i=1}^n{\dd(\gamma(t_{i-1}), \gamma(t_{i}))}},
$$
where the supremum is taken over all finite collections
of points $0=t_0<t_1<\cdots<t_n=1$. A path is called \emph{rectifiable}
if its length is finite. The length of a
restriction $\gamma|_{[a,b]}$ will be denoted $L(\gamma,a,b)$. 

When focusing on geometric properties of paths, it is sometimes convenient not to
distinguish paths that differ only up to a change of variable.
To this end, we define a \emph{curve} as an equivalence class of the
equivalence relation for which paths $\gamma_1$ and $\gamma_2$
are equivalent if $\gamma_1(\varphi_1(t)) = \gamma_2(\varphi_2(t))$, where
$\varphi_i:[0,1]\to[0,1]$, $i=1,2$, are continuous, nondecreasing
functions (see \cite{burago2001} for details). A particular path within a curve is
called a parametrization of that curve. Paths representing the same
curve are re-parametrizations of each other. Such paths have the same
image and the same length, allowing us to define these concepts for
curves. If a curve is rectifiable then it has the constant speed
parametrization, which is the path $\gamma$ such that $L(\gamma, t_0,
t_1)=L(\gamma)(t_1-t_0)$.

In the case of loops, it is further often useful to not fix the starting point.
Hence, the notion of a curve has to be slightly modified.  We define a
\emph{free loop} as an equivalence class of the equivalence relation for
which loops $\gamma_1,\gamma_2:\S^1\to\X$ are equivalent if
$\gamma_1(\varphi_1(t)) = \gamma_2(\varphi_2(t))$, where
$\varphi_i:\S^1\to\S^1$, $i=1,2$, are orientation preserving homeomorphisms. 

Once again, loops
representing the same free loop have the same image and length, and
rectifiable free loops admit a constant speed parametrization. If
$\hat\gamma$ is a free loop (or a curve) we define
$L(\hat\gamma)=L(\gamma)$, where $\gamma$ is a representation of
$\hat\gamma$.

One of the central concepts in the topology and geometry of paths is
homotopy.  Two paths $\gamma_0$ and $\gamma_1$ such that
$\gamma_0(0)=\gamma_1(0)=x\in\X$ and $\gamma_0(1)=\gamma_1(1)=y\in\X$
are said to be homotopic if there exists a continuous map $H:I\times
[0,1]\to\X$ such that $H(\cdot, 0)=\gamma_0$, $H(\cdot, 1)=\gamma_1$,
and $H(0,t)=x, H(1,t)=y$ for all $t\in[0,1]$. Intuitively, two paths are
homotopic if one can be continuously deformed into the other keeping the
endpoints fixed. It is useful to note that two representations of the
same curve are homotopic.

The homotopy keeps the starting point fixed, which may be undesirable
when dealing with loops. In this case we we need to use the \emph{free
homotopy}. More precisely, loops $\gamma_0$ and $\gamma_1$ are said to
be freely homotopic if there exists a continuous map $H:I\times
[0,1]\to\X$ such that $H(\cdot, 0)=\gamma_0$, $H(\cdot, 1)=\gamma_1$,
and $H(0,t)=H(1,t)$ for all $t\in [0,1]$.  Similarly to the case of
curves, two representations of the same free loop are freely homotopic.
Also, being freely homotopic is an equivalence relation, and an
equivalence class of freely homotopic loops is called a free homotopy
class. Such a class is called trivial if it contains a constant loop
(i.e. a point). Loops within the trivial free homotopy class are called
contractible. A contractible loop is actually homotopic to a constant
loop.

We denote the space of paths in $\X$ by $\Omega(\X)$ and endow it with
the $C^0$ metric, which we denote by $\rho$. That is, given $\gamma_0,
\gamma_1\in\Omega(\X)$, the distance between them is defined by
$\rho(\gamma_0, \gamma_1) =
\sup_{t\in[0,1]}\dd(\gamma_0(t),\gamma_1(t))$. The subspace of
$\Omega(\X)$ consisting of loops will be denoted by $\LL(\X)$.  We may
also consider the space of curves in $\X$, which we denote by
$\hat\Omega(\X)$ and the space of free loops, $\hat\LL(\X)$. The maps
$\pi_{\Omega}:\Omega(\X)\to\hat\Omega(\X)$ and
$\pi_{\LL}:\LL(\X)\to\hat\LL(\X)$ will denote the corresponding canonical
projections.  We endow both $\hat\Omega(\X)$ and $\hat\LL(\X)$ with a
metric.  The distance between
$\hat\gamma_0,\hat\gamma_1\in\hat\Omega(\X)$ is defined as
$\hat{\rho}_{\Omega}(\hat\gamma_0, \hat\gamma_1) =
\inf_{\gamma_i\in\pi_{\Omega}^{-1}(\hat\gamma_i)}\rho(\gamma_0,\gamma_1)$.
Similarly, the distance between $\hat\gamma_0,\hat\gamma_1\in\hat\LL(\X)$
is defined as $\hat{\rho}_{\LL}(\hat\gamma_0, \hat\gamma_1) =
\inf_{\gamma_i\in\pi_{\LL}^{-1}(\hat\gamma_i)}\rho(\gamma_0,\gamma_1)$.

\subsection{Paths and loops in a punctured plane}
The concrete metric space that we consider in this paper is a
multi-punctured plane, $X=\R^2\setminus Z$, $Z=\{z_1,\ldots,z_K\}$,
$z_i\in\R^2$, with the standard Euclidean metric, $\dd(x,y) = \|x-y\|$,
where $\|\cdot\|$ is the Euclidean norm. 
By $\reach(Z)$ we denote half the minimum
distance between the punctures, $\reach(Z)=\frac{1}{2}\min_{z,w\in
Z}\|z-w\|$.  Also, it will be convenient to define
$X^{\delta}=\R^2\setminus\cup_{i=1}^{K}{B_{\delta}(z_i)}$,
$\delta>0$, where $B_{\delta}(z_i)$ denotes  an open ball of radius
$\delta$ centered at $z_i$. For $\delta<\reach(Z)$, $X^\delta$ is homotopy equivalent to $X$.

A free homotopy class shall be regarded as a connected component of the
space of loops in $X$. Given sets $A\subset B\subset\R^2$, we shall regard
$\Omega(A)$ as a subset of $\Omega(B)$, and $\LL(A)$ as a subset of
$\LL(B)$. In particular, we have $\LL(X^{\delta})\subset\LL(X)$. Throughout
the rest of the paper, $\homt(X)\subset\LL(X)$ will denote a fixed,
nontrivial free homotopy class of $\LL(X)$, and
$\homt(X^{\delta})\subset\LL(X^{\delta})$ will be the free homotopy class
of $\LL(X^{\delta})$ such that $\homt(X^{\delta})\subset\homt(X)$.
Notice that $\homt(X^{\delta})$ is well defined if
$\delta<\reach(Z)$, which we assume hereafter. We also define
$\hat\homt(X)=\pi(\homt(X))$,
$\hat\homt(X^{\delta})=\pi_{\LL}(\homt(X^{\delta}))$.

Loosely speaking, our goal is to show that a loop chosen ``uniformly at
random'' in $\hat\homt(X)$ is extremely likely to be very close to the
shortest loop (essentially, unique) in that class. 

To make this statement more precise, we need to define an
appropriate probability measure on $\hat\homt(X)$. Such a probability
measure can be obtained as a push forward of a probability measure on
$\homt(X)$. In fact, we shall consider a sequence of probability
measures on $\homt(X)$, each supported on an increasingly finer finite
dimensional approximation of loops in $\homt(X)$. The result for
$\hat\homt(X)$ will then be obtained as a corollary of a stronger result
for $\homt(X)$. In what follows, it will be convenient to regard
$\Omega(X)$ (as well as any of its subsets, e.g. a homotopy class) as a
subset of $\Omega(\R^2)$.

A path $\gamma$ in $\R^2$ is called linear with endpoints
$x,y\in\R^2$ if $\gamma(t)=x+t(y-x)$. Such a path will be denoted by
$[x,y]$.  Clearly, the length of $[x,y]$ is just $\dd(x,y)$. We say that
$\gamma$ is a piecewise linear path if it is a concatenation of finitely
many linear paths. Each linear path of such a concatenation is called an
edge of $\gamma$, and an endpoint of an edge is called a vertex of
$\gamma$. It is easy to see that the length of a piecewise linear path
is just the sum of its edge lengths. We will denote the set of
piecewise linear paths in $Y\subset\R^2$ by $\Omega_{PL}(Y)$; the
space of piecewise linear loops in $Y$ will be denoted $\LL_{PL}(Y)$.
Notice that given a piecewise linear path one can ``close'' it by adding an
edge between the first and the last vertices. Alternatively, one can
``open'' a piecewise linear loop by removing
the last edge. We will use this fact, and we define
$\iota:\LL_{PL}(\R^2)\to\Omega_{PL}(\R^2)$ by
$\iota(e_1\overset{\alpha_1}{\cdot}\ldots\overset{\alpha_{m}}{\cdot}
e_{m+1})=e_1\overset{\beta_1}{\cdot}\ldots\overset{\beta_{m-1}}{\cdot} e_{m}$,
where $e_i$, $i=1,\ldots,m$, are linear paths and
$\beta_i=\alpha_i/\sum_{j=1}^{m-1}{\alpha_i}$.

The following result shows that piecewise linear paths form a dense set:
\begin{prp}
    \label{prp:path_approx}
    Let $\gamma\in\Omega(\R^2)$. Then $\forall\e>0$ there
    exists $\gamma_{PL}\in\Omega_{PL}(\R^2)$ such
    that $\rho(\gamma, \gamma_{PL})<\e$.  Moreover, $\gamma_{PL}$ can be
    chosen such that $\gamma_{PL}$ is a loop if $\gamma$ is a loop,
    $\gamma_{PL}(t)=\gamma(t)$ if $\gamma_{PL}(t)$ is a vertex,
    and each edge of $\gamma_{PL}$
    has traversal time $\frac{1}{m}$, where $m$ is the number of edges.
\end{prp}

A curve or a free loop is called piecewise linearizable if it possesses a
piecewise linear parametrization. A piecewise linear curve or free loop
is completely determined by its vertices. Hence, there is a
correspondence between $\R^{2(n+1)}$ and piecewise linear curves (or
free loops) in $\R^2$ with $n+1$ vertices. More precisely, one can take
$\mathbf{v}=(v_0,\ldots, v_n)\in\R^{2(n+1)}$ to correspond to the curve
represented by a concatenation of linear paths $[v_{i-1}, v_i]$,
$i=1,\ldots,n$. If we also concatenate $[v_n, v_0]$, then $\mathbf{v}$
corresponds to the resulting free loop.

Given the starting point $v_0$ and approximation scale ${n}$, we define the
{\em finite-dimensional approximations} for the
curve and free loop yields maps as
$\Psi_n:\R^{2(n+1)}\to\Omega(\R^2)$ and
$\Phi_n:\R^{2(n+1)}\to\LL(\R^2)$, respectively.

We will also use an alternative correspondence between
$\R^{2(n+1)}$ and piecewise linear curves and free loops. It is obtained
by letting $\mathbf{w}=(v_0, x_1,\ldots,x_n)\subset\R^{2(n+1)}$
correspond to the curve (or free loop) with vertices $v_0,
S_1,\ldots,S_n$, where $S_k=v_0+\sum_{i=1}^{k}{x_i}$, $k=1,\ldots,n$.
The corresponding maps from $\R^{2(n+1)}$ into $\Omega(\R^2)$ and
$\LL(\R^2)$ are compositions of $\Psi_n$
and $\Phi_n$ with the homeomorphism $f:\R^{2(n+1)}\to\R^{2(n+1)}$
defined by $f(v_0,x_1,\ldots,x_n)=(v_0, S_1,\ldots,S_n)$, that is, we
consider $\tilde\Psi_n=\Psi_n\circ f$ and $\tilde\Phi_n=\Phi_n\circ f$.

Now, if we take the uniform probability measure on the appropriate
subset $G_n\subset\R^{2(n+1)}$ we can push it forward to $\homt(X)$
using $\Phi_n$, and then further to $\hat\homt(X)$ using $\pi_{\LL}$.  Of
course, $G_n$ should be bounded. Also, as $n$ increases, we would like
the image of $G_n$ under $\pi_{\LL}\circ\Phi_n$ to provide an
increasingly finer approximation of loops in $\hat\homt(X)$. To achieve
boundedness we need to restrict ourselves to loops of bounded length.
Hence, let $R>0$, and let $\hat\homt^R(X)$ be the set of free loops in
$\hat\homt(X)$ with length less than $R$. We choose $R$ sufficiently
large, so that $\hat\homt^R(X)\neq\emptyset$. Notice that Proposition
\ref{prp:path_approx} implies that any free loop in $\hat\homt^R(X)$ can
be approximated by a piecewise linear free loop, and this approximation
improves with decreasing edge length.  Therefore, we define $G_n$ as
follows:
\begin{multline*}
G_n =
\left\{\mathbf{x}=(x_0,\ldots,x_n)\in\R^{2(n+1)} :\right.\\
\left.    \Phi_n(x)\in\homt(X), \|x_0-x_{n}\|<\frac{R}{n+1},
\|x_{i}-x_{i-1}\|<\frac{R}{n+1},i=1,\ldots,n
\right\}
\end{multline*}
Let $\hat\homt^R_n(X)=\pi_{\LL}\circ\Phi_n(G_n)$. It is the set of piecewise
linear free loops in $\hat\homt(X)$ with $n+1$ vertices and edge lengths
less than $\frac{R}{n+1}$. Also, let $\homt^R_n(X)=\Phi_n(G_n)$, which
is the set of piecewise linear loops in $\homt(X)$ with $n+1$ vertices
whose edges have traversal time $\frac{1}{n+1}$ and length less than $\frac{R}{n+1}$.
Clearly, such loops have speed strictly bounded by $R$. We let
$\homt^R(X)$ denote the set of all loops in $\homt(X)$ with speed strictly
bounded by $R$ and notice that $\hat\homt^R(X)=\pi_{\LL}(\homt^R(X))$.

Define $\nu_n$ to be the push forward under $\Phi_n$ of the uniform
probability measure on $G_n$, and let $\hat\nu_n$ be the push forward of
$\nu_n$ under $\pi_{\LL}$. We can now state our goal more precisely (although
still somewhat informally): we want to show that $\hat\nu_n$ becomes
overwhelmingly concentrated around the shortest loop as $n\to\infty$. We
make this statement completely rigorous in the next section.

\subsection{Random paths and Mogulskii's theorem}
The above definition of $G_n$ allows for an alternative description of
$\nu_n$ which is better amenable to analysis.  Denote by
$B_R\subset\R^2$ the disk of radius $R$ centered at the origin and by
$A_n\subset\R^2$ the projection of $G_n$ onto the first two coordinates.
Note that $A_n\subset A_{n+1}$, and $A=\cup_{n}{A_n}$ is bounded. 
Let $\mu$ be the uniform probability measure
on $B_R$, where for ease of notation we suppressed the explicit
dependence on $R$, and let $\upsilon_n$
be the uniform probability measure on $A_n$. Suppose that $V_n$ is a random
variable with the probability law $\upsilon_n$, $X_1,\ldots,X_n$ are
i.i.d. random variables with the probability law $\mu$, and consider the
random piecewise linear path $\tilde\Psi_n\left(V_n,
\frac{X_1}{n},\ldots,\frac{X_n}{n}\right)$. Let $\mu_n$ be the
probability law of such a path. Then given $\Gamma\subset\homt(X)$ we
have
$\nu_n(\Gamma)=\frac{\mu_n(\iota(\Gamma\cap\homt^R_n(X)))}{\mu_n(\iota(\homt^R_n(X)))}$.
For convenience, $\nu_n$, $\mu$, $\mu_n$, and $\upsilon_n$ will retain
the aforementioned meaning throughout the paper.

With such a set-up we are in the position to employ the powerful
machinery of the large deviation theory, in particular the Mogulskii's
Theorem. First, we need to introduce a few more concepts and results. A
\emph{rate function} on a topological space $\Y$ is a lower
semicontinuous map $I:\Y\to [0,\infty]$ such that its sublevel sets,
$\{y\in \Y|I(y)\leq\alpha\}$, $\alpha\in [0,\infty]$, are closed. A rate
function is called \emph{good} if its sublevel sets are compact. By
$\D_I$ we will denote the \emph{effective domain} of the rate function
$I$, that is, $\D_I=\{y\in\Y| I(y)<\infty\}$.

Taking into account our alternative description of $\nu_n$, let
$\Lambda$ denote the \emph{logarithmic moment generating function}
associated with $\mu$, that is $\Lambda(\eta) = \log{\E(e^{<X,\eta>})}$,
where $\E(\cdot)$ denotes the expectation, $X$ has probability law
$\mu$, and $<\cdot, \cdot>$ denotes the inner product. Define
$\Lambda^*$ to be the Fenchel-Legendre transform of $\Lambda$, that is
$\Lambda^*(x) = \sup_{\eta}{[<x,\eta>-\Lambda(\eta)]}$. The following
proposition summarizes the properties of $\Lambda$ and $\Lambda^*$:
\begin{prp}
    \label{prp:rate_func}
    \begin{enumerate}
        \item \label{i:1} $\Lambda$ is a strictly convex,
            everywhere differentiable function. 
        \item \label{i:2} $\Lambda^*$ is a good strictly convex
            rate function.
        \item \label{i:3} If $y=\nabla\Lambda(\eta)$ then $\Lambda^*(y) =
            <\eta,y>-\Lambda(\eta)$.
        \item \label{i:4} Both $\Lambda$ and $\Lambda^*$ are invariant
            under rotations around the origin,
            $\D_{\Lambda^*} = B_R$, and $\forall y\in B_R$ $\exists
            \eta\in\R^2$ such that $y=\nabla\Lambda(\eta)$.
    \end{enumerate}
\end{prp}

Recall that a map $\phi:[0,1]\to\R^2$ is called absolutely continuous if
$\forall \e>0$ $\exists\delta>0$ such that $\sum_{i=1}^{m}{\d(\phi(y_i,
x_i))}<\e$ for every finite collection of disjoint intervals
$(x_i,y_i)\subset [0,1]$, $i=1,\ldots,m$, such that
$\sum_{i=1}^{m}{|y_i-x_i|}<\delta$. 

We will denote the space of
absolutely continuous paths and loops in $Y\subset\R^2$ by
$\Omega_{\AC}(Y)$ and $\LL_{\AC}(Y)$, respectively. It is useful to note
that if $\gamma\in\Omega_{\AC}(\R^2)$ then it is differentiable almost
everywhere and $\LL(\gamma,a,b) = \int_{a}^{b}{\|\gamma'(t)\| dt}$, where
$[a,b]\subset[0,1]$ and $\gamma'(t)$ denote the derivative of $\gamma$ at $t$.

We are now ready to state the Mogulskii's theorem:
\begin{thm}[Mogulskii]
    \label{thm:mogulskii}
    Let $\tilde\mu_n$ denote the probability law of the random path
    $\tilde\Psi\left(0,\frac{X_1}{n},\ldots,\frac{X_n}{n}\right)$, where $X_0,\ldots,X_n$ are i.i.d.
    random variables with the probability law $\mu$. Then the function
    $I_0:\Omega(\R^2)\to[0,\infty]$ defined by
    $$
    I_0(\phi) = \left\{
        \begin{array}{cl}
            \int_{0}^{1}{\Lambda^*(\phi'(t)) dt},&\quad \text{if
        }\phi\in\Omega_{\AC}(\R^2), \phi(0)=0\\
        \infty,&\quad\text{otherwise}
        \end{array}
        \right.
    $$
    is a good rate function, and for any Borel set
    $\Gamma\subset\Omega(\R^2)$ we have
    $$
    -\inf_{x\in\Gamma^{\circ}}{I_0(x)} \leq
    \liminf_{n\to\infty}{\frac{1}{n}\log{\tilde\mu_n(\Gamma)}}\leq
    \limsup_{n\to\infty}{\frac{1}{n}\log{\tilde\mu_n(\Gamma)}}\leq
    -\inf_{x\in\cl\Gamma}{I_0(x)},
    $$
    where $\Gamma^{\circ}$ denotes the interior of $\Gamma$ and
    $\cl\Gamma$ denotes the closure of $\Gamma$.
\end{thm}
The same result holds also in the subspace $\Omega_0(\R^2)$ consisting
only of paths starting at the origin (see \cite{dembo2009} for details), or at any other point.

More generally, we can prove a version of the Mogulskii's theorem where
the starting point is chosen uniformly at random.
\begin{thm}
    \label{thm:mogulskii2}
    Suppose that $E_n\subset\R^2$ are open, $E_n\subset
    E_{n+1}$, and $E=\cup_n{E_n}$ is bounded. Let $\tilde\upsilon_n$ be
    the uniform probability measure on $E_n$, and let $V_n$ be a random variable
    with the probability law $\tilde\upsilon_n$.
    Denote by $\tilde\mu_n$ the probability law of the random path
    $\tilde\Psi\left(V_n,\frac{X_1}{n},\ldots,\frac{X_n}{n}\right)$, where $X_0,\ldots,X_n$ are i.i.d.
    random variables with the probability law $\mu$. Then the function
    $I_E:\Omega(\R^2)\to[0,\infty]$ defined by
    $$
    I_E(\phi) = \left\{
        \begin{array}{cl}
            \int_{0}^{1}{\Lambda^*(\phi'(t)) dt},&\quad \text{if
        }\phi\in\Omega_{\AC}(\R^2), \phi(0)\in \cl E\\
        \infty,&\quad\text{otherwise}
        \end{array}
        \right.
    $$
    where $\cl E$ denotes the closure of $E$, is a good rate function, and for any Borel set
    $\Gamma\subset\Omega(\R^2)$ we have
    $$
    -\inf_{x\in\Gamma^{\circ}}{I_E(x)} \leq
    \liminf_{n\to\infty}{\frac{1}{n}\log{\tilde\mu_n(\Gamma)}}\leq
    \limsup_{n\to\infty}{\frac{1}{n}\log{\tilde\mu_n(\Gamma)}}\leq
    -\inf_{x\in\cl\Gamma}{I_E(x)},
    $$
    where $\Gamma^{\circ}$ denotes the interior of $\Gamma$ and
    $\cl\Gamma$ denotes the closure of $\Gamma$.
\end{thm}

We shall refer to Theorem \ref{thm:mogulskii2} as \emph{untethered}
Mogulskii's theorem. If $E_n=\cup_{k=1}^n A_k$, the projection of $G_n$ onto the
first two coordinates, then we denote the corresponding $I_E$ simply by
$I$.  It is useful to notice that if $\phi\in\homt^R(X)$ then
$\phi(0)\in A=\cup_{n}{A_n}$.

Our particular choice of the probability law $\mu$ leads to several
useful properties of the rate functions $I_0$ and $I_E$.
\begin{prp}
    \label{prp:I_prop}
    Let $J$ be either $I_0$ or $I_E$.
    \begin{enumerate}
        \item \label{ii:1}
            $\D_{J}$ $\subset$
            $\{\phi\in\Omega_{\AC}(\R^2) : \|\phi'(t)\|<R \text{ a.e. on } [0,1]\}$
            $\subset$ $\Omega_{\AC}(\R^2)\cap\Omega^R(\R^2)$,
            where $\Omega^R(\R^2)$ denotes the set of
            paths with Lipschitz constant bounded by $R$.
        \item \label{ii:2}
            Let $\gamma\in\D_J$ be a constant speed parametrization of a curve
            or a free loop $\hat\gamma$, and let $\Gamma$ be the set of
            all parametrizations of $\hat\gamma$. Then
            $$
            \inf_{\phi\in\Gamma}{J(\phi)} = J(\gamma)
            $$
        \item \label{ii:3}
            Suppose that $\gamma\in\D_J$ is a (non-constant) path with
            constant speed parametrization, and let $\phi\in\D_J$ be a path such that
            $L(\phi)\geq L(\gamma)+\e$. Then there exists a constant
            $c>0$, depending on $\gamma$, such that
            $J(\phi)-J(\gamma)\geq c\e$.
    \end{enumerate}
\end{prp}

\subsection{Path localization results}
Mogulskii's theorem and the properties of the rate functions $I_0$ and
$I_E$ allow us to investigate the behavior of $\tilde\mu_n$ when restricted to a
particular set $\Gamma\subset\Omega(\R^2)$. For example, let $\Gamma$
consist of paths starting at the origin and ending within
the closed ball $\cl B_r(a)=\{x\in\R^2:\dd(x,a)\leq r\}$, $a\in\R^2$.
Suppose also that $0\notin \cl B_r(a)$ and $r+\|a\|<R$. Then the
following holds:
\begin{cor}
\label{cor:straight_path}
Let $x^*\in \cl B_r(a)$ be the point closest to the origin, and let
$\hat\gamma^*=\pi_{\Omega}([0,x^*])$. Take $\delta>0$ and let
$\hat\Gamma_{\delta}=\{\hat\gamma\in\pi_{\Omega}(\Gamma)|\hat\rho_{\Omega}(\hat\gamma,
\hat\gamma^*)\geq\delta\}$,
$\Gamma_{\delta}=\pi_{\Omega}^{-1}(\hat\Gamma_{\delta})\cap\Gamma$. Then there
exists a constant $c>0$ (depending on $x^*$ and $r$) such that
$$
\limsup_{n\to\infty}{\frac{1}{n}\log{\frac{\tilde\mu_n(\Gamma_{\delta})}{\tilde\mu_n(\Gamma)}}}\leq
-c\delta^2
$$
\end{cor}
In other words, $\tilde\mu_n$ restricted to the above $\Gamma$ become
overwhelmingly concentrated around the shortest paths.

It is reasonable to expect a similar concentration result for $\nu_n$.
Unfortunately, as follows from an earlier discussion, investigating the
behavior of $\nu_n$ requires us to consider ratios of the form
$\frac{\mu_n(Q_n)}{\mu_n(P_n)}$, $Q_n\subset P_n$, rather than
$\frac{\mu_n(Q)}{\mu_n(P)}$ for fixed $Q\subset P$. Hence, a direct
application of Mogulskii's theorem is not feasible. In the next section
we detail our approach to overcome this difficulty.
\section{Typical loops in $\homt(X)$}
\label{sec:loop_ldp}
Before we rigorously state our main result we
need to take care of a small technicality.
Unlike the situation in Corollary \ref{cor:straight_path}, where the minimizing path belongs to
the set under consideration, $\homt(X)$ does {\em not} contain any loop minimizing the rate. However, $L(\cdot)$ does attain its infimum on
$\cl\homt(X)$, the closure of $\homt(X)$ in $\Omega(\R^2)$, and
consequently on $\pi_{\LL}(\cl\homt(X))$.
Moreover, the shortest loop in
$\cl\homt(X)$ is unique up to reparametrization and is, in fact,
piecewise linear.

\begin{lem}
\label{lem:shortest_loop}
$\pi_{\LL}(\cl\homt(X))$ contains a unique free loop of the
shortest length. Moreover, this shortest free loop is piecewise linear with
vertices in $Z$.
\end{lem}

We let $\hat\gamma^*$ denote the shortest
free loop in $\pi_{\LL}(\cl\homt(X))$. Our main result shows that
$\hat\nu_n$ become overwhelmingly concentrated around $\hat\gamma^*$ as
$n\to\infty$.
\begin{thm}
\label{thm:main_free}
For each $\delta>0$ we have
$$
\limsup_{n\to\infty}{\frac{1}{n}\log{\hat\nu_n(\hat\Gamma_{\delta})}}\leq
-c\delta^2,
$$
where $c>0$ is a constant, and
$\hat\Gamma_{\delta}=\{\hat\gamma\in\hat\homt(X)|\hat\rho_{\LL}(\hat\gamma,
\hat\gamma^*)\geq\delta\}$.
\end{thm}
Since $\hat\nu_n$ is a push forward of $\nu_n$ under $\pi_{\LL}$, Theorem
\ref{thm:main_free} is an immediate corollary of the following result.
\begin{thm}
\label{thm:main}
For each $\delta>0$ we have
$$
\limsup_{n\to\infty}{\frac{1}{n}\log{\nu_n(\Gamma_{\delta})}}\leq
-c\delta^2,
$$
where $c>0$ is a constant, and
$\Gamma_{\delta}=\pi_{\LL}^{-1}(\hat\Gamma_{\delta})$.
\end{thm}
The proof of the above theorem relies on Proposition \ref{prp:main}
below, which can be regarded as a variation of the Mogulskii's theorem.
Recall that by Proposition \ref{prp:I_prop} $I(\cdot)$ attains the same
value for any constant speed parametrization of $\hat\gamma^*$. Let us
denote this value by $I^*$.
\begin{prp}
    \label{prp:main}
    For any Borel subset $\Gamma\subset\homt(X)$ we have
    $$
    -(\inf_{\gamma\in\Gamma^{\circ}}{I(\gamma)} - I^*) \leq
    \liminf_{n\to\infty}{\frac{1}{n}\log{\nu_n(\Gamma)}}\leq
    \limsup_{n\to\infty}{\frac{1}{n}\log{\nu_n(\Gamma)}}\leq
    -(\inf_{\gamma\in\cl\Gamma}{I(\gamma)}-I^*),
    $$
    where $\Gamma^{\circ}$ and $\cl\Gamma$ denote the interior
    the closure of $\Gamma$ in $\Omega(\R^2)$, respectively.
\end{prp}
The key ingredients in the proof of this proposition are the
untethered Mogulskii's theorem and the following lemma, which is of
independent interest in itself:
\begin{lem}
    \label{lem:lower_bound}
    Let $\Gamma\subset\homt(X)$ be open, and let
    $\Gamma_n=\iota(\Gamma\cap\homt^R_n(X))$. Then
    $$
    -\inf_{\gamma\in\Gamma}{I(\gamma)} \leq
    \liminf_{n\to\infty}{\frac{1}{n}\log{\mu_n(\Gamma_n)}}
    $$
\end{lem}

As mentioned in the Introduction, the proofs of these results are
postponed till Section \ref{sec:proofs}.

\section{Sampling in $G_n$}
\label{sec:sampling}
Any practical application of the results from the previous section requires the
ability to sample from $\nu_n$. In this section we show that a standard Markov
Chain Monte Carlo (MCMC) techique can do the job. A comprehansive description
of Markov chains and MCMC methods can be found in \cite{robert2005, meyn2009}
and references therein. Here, we shall limit ourselves to describing and
justifying a particular sampling procedure, providing definitions of only some
concepts. 

\subsection{The sampling algorithm}
As any MCMC method, the sampling algorithm that we propose is based on
constructing an ergodic Markov chain on $G_n$ whose limiting distribution is
$\nu_n$.  For convenience, we shall now fix $n$ and let $G=G_n$, $\nu=\nu_n$,
$\e=\frac{R}{n}$. Also, we assume that $n$ is large enough so that
$G_{n}\neq\emptyset$ and $R/n<\reach(Z)$. The algorithm starts with an arbitrary initial state
$\vs_0\in G$. Given that the chain is in state $\vs_i\in G$, $i\geq 0$,
the next state, $\vs_{i+1}$ is generated as follows. Suppose that
$\vs_i=(v_0,\ldots,v_n)\in\R^{2(n+1)}$ (in other words, $v_0,\ldots,v_n$ are
the vertices of the corresponding loop), and let $v_{-1}=v_n, v_{n+1}=v_0$.
Select $k$ uniformly at random from $\{0,\ldots,n\}$. Let $D$ be the
intersections of two open balls of radius $\e$ centered at $v_{k-1}$ and
$v_{k+1}$. The idea is to choose the next state by moving $v_k$ to a randomly
chosen point in $D$, but we have to make sure that we do not change the free
homotopy class of the corresponding loop. Notice that $D$ may contain at most
one puncture. If $Z\cap D=\emptyset$ we let $E=D$. If some $z_{j}\in D$ we
let $H_1$ be the open half space supported by the line through $v_{k-1}$ and
$z_i$ and not containing $v_{k+1}$, $H_2$ be the open half space supported by
the line through $v_{k+1}$ and $z_i$ and not containing $v_{k-1}$, and
$H=H_1\cap H_2$. Then if $v_{k}\in H$ we let $E=D\cap H$, otherwise, $E=D\setminus
\cl H$ (see Figure \ref{fig:next_state}). Choose $\bar v_{k}$ uniformly at
random from $E$ and set the next state $\vs_{i+1}=(v_0,\ldots,v_{k-1},\bar
v_{k},v_{k+1},\ldots,v_n)$.

\begin{figure}[htb!]
    \centering
    \includegraphics[width=0.8\textwidth]{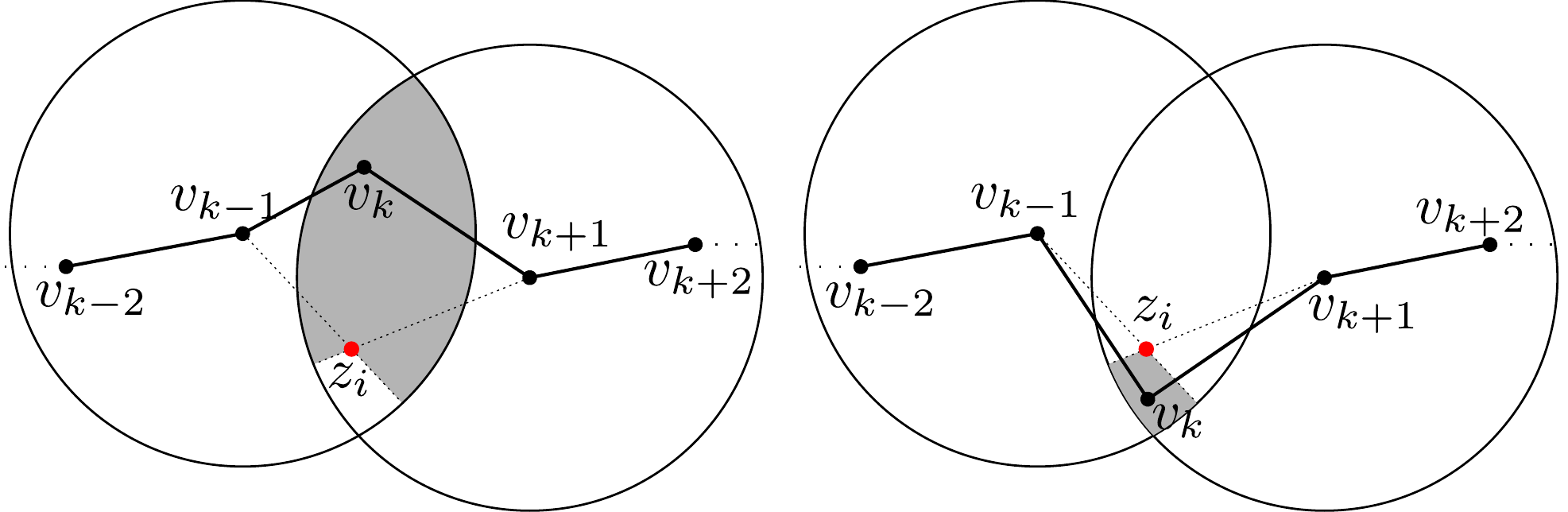}
    \caption{
        \label{fig:next_state}
        The next vertex position during the MCMC procedure is selected
        uniformly from the shaded region. Red points indicate
        punctures; left and right show two different relative positions
        of the vertex being moved and a puncture.
    }
\end{figure}

The above algorithm can be classified as a Metropolis-within-Gibbs algorithm
(see e.g. \cite{roberts2006}), and it follows from standard results that
sequence $\{\vs_i\}$ is a Markov chain whose stationary
distribution is $\nu$. Of course, we also need to show that the chain converges
to $\nu$.  To make this statement more precise, let $P(\vb,\cdot)$,
$\vb\in G$, be the corresponding transition probability measure, i.e.
$P(\vb, A)=\P(\vs_{i+1}\in A|\vs_i=\vb)$, where $A$ is a Borel
subset of $G$ (see \cite{roberts2006} for details).  Denote by
$P^m(\vb,\cdot)$ the probability law of the $m$-th element of the chain when
starting at $\vb$, i.e.  $P^m(\vb, A)=\P(\vs_{m}\in
A|\vs_0=\vb)$. The total variation norm of a signed measure $\mu$ is
defined by $\|\mu\|=\sup_{A\in\mathcal{M}}{|\mu(A)|}$, where $\mathcal{M}$
denotes the collection of $\mu$-measurable sets. We would like to show that
$$
\lim_{m\to\infty}\|P^m(\vb,\cdot)-\nu\| = 0,\quad\forall\vb\in G,
$$
which implies that, regardless of the initial state, our algorithm generates
samples from an almost uniform distribution on $G$ after a large enough number
of steps.

It is well known (see e.g. \cite{robert2005, meyn2009}) that the above
convergence result holds if our Markov chain is $\nu$-irreducible, aperiodic, and
Harris recurrent. $\nu$-irreducibility means that for any Borel set $A\subset
G$ such that $\nu(A)>0$ there exists $m\in\N$ such that $P^m(\vb,A)>0$ for
all $\vb\in G$. Aperiodicity means that if $S_1,\ldots,S_k\subset G$ are
disjoint Borel sets such that $\nu(S_j)>0$ and $\P(\vb, S_{j+1})=1\;\forall
\vb\in S_{j}$, where $j=1,\ldots,k$, $S_{k+1}=S_1$, then $k=1$. Finally,
Harris recurrence means that for any Borel set $A\subset G$ such that
$\nu(A)>0$ we have $\P(\vs_m\in A\,i.o.|\vs_0=\vb)=1\;\forall\vb\in
G$, where i.o. stands for ``infinitely often''.

\begin{prp}\label{prp:sampling_main}
    \begin{enumerate}
        \item \label{i:samp_main1} Suppose that $G$ is path connected.
            Then the Markov chain $\{\vs_i\}$ is $\nu$-irreducible,
            aperiodic and Harris recurrent.
        \item \label{i:samp_main2} $G_n$ is path connected for large
            enough $n$.
    \end{enumerate}
\end{prp}

The proof of the above proposition is provided in a separate subsection
of Section \ref{sec:proofs}.

\subsection{Numerical simulations}
To illustrate the behavior of our algorithm we have performed some
numerical simulations. For simplicity, the actual implementation of the
algorithm slightly deviates from the description given above. In
particular, the vertex to move at each step is chosen as follows. We
generate a random permutation of indices, $\{i_0,\ldots,i_n\}$, and then
move vertices according to their order in the permutation until all the
vertices have been moved. After that a new random permutation is
generated and the process repeats. In addition, the new position of the
vertex being moved is generated by subsampling the allowable region.  It
is not difficult to show (using essentially the same argument) that the
resulting Markov chain is still $\nu_n$-irreducible, aperiodic and
Harris recurrent, and hence converges (in the total variation norm) to
the uniform distribution on $G_n$.

Our simulations were done for $n=599$ (i.e. loops have $600$
vertices). We performed $2\cdot 10^6$ iterations (where by an iteration
we mean a single pass over all vertices in a random permutation), saving
a loop after each $100$ iterations. Out of saved loops we selected
$50$ last ones. As a proxy for the density of the loop distribution, we
computed the standard kernel density estimation for their vertex
positions. Also, we computed a ``mean'' free loop. This computation was
done by cyclically permuting vertices to minimize the distance between
the corresponding elements of $\R^{2(n+1)}$ and then computing the mean
position for each vertex. It is important to note that such a
computation does not preserve the homotopy class, but it does
provide useful geometric information.

Figure \ref{fig:sim_sq} shows the results of the above computations for
a plane with four punctures, $Z=\{1.35,-1.35\}\times\{1.35,-1.35\}$, and the
free homotopy class of a circle containing all the punctures. We chose
the upper bound on the loop length $R=20$. The shortest free loop,
$\gamma^*$, is in this case the square with vertices in $Z$. It is
evident from the figure that the uniform distribution in
$\homt^R_n(X)$ is nicely concentrated around $\gamma^*$, and the mean
free loop of only $50$ samples has a fairly regular shape close to
$\gamma^*$. Of course, each individual sample has a much more irregular
shape. 

Similar results can be seen in Figure \ref{fig:sim_bf}, where the
computations were done for the plane with punctures $z_1=(-1.3, 0.6)$,
$z_2=(1.3, 0.6)$, $z_3=(1.3, -0.6)$, $z_4=(-1.3, -0.6)$, and the
homotopy class of a lemniscate, as shown in the plot
\ref{fig:sim_bf}(a).  The upper bound on the loop length is again
$R=20$. The shortest free loop, $\gamma^*$, is in this case a ``bow tie''
quadrilateral $z_1 z_3 z_2 z_4$. 

\begin{figure}[!htb]
\centering
\includegraphics[width=0.246\textwidth]{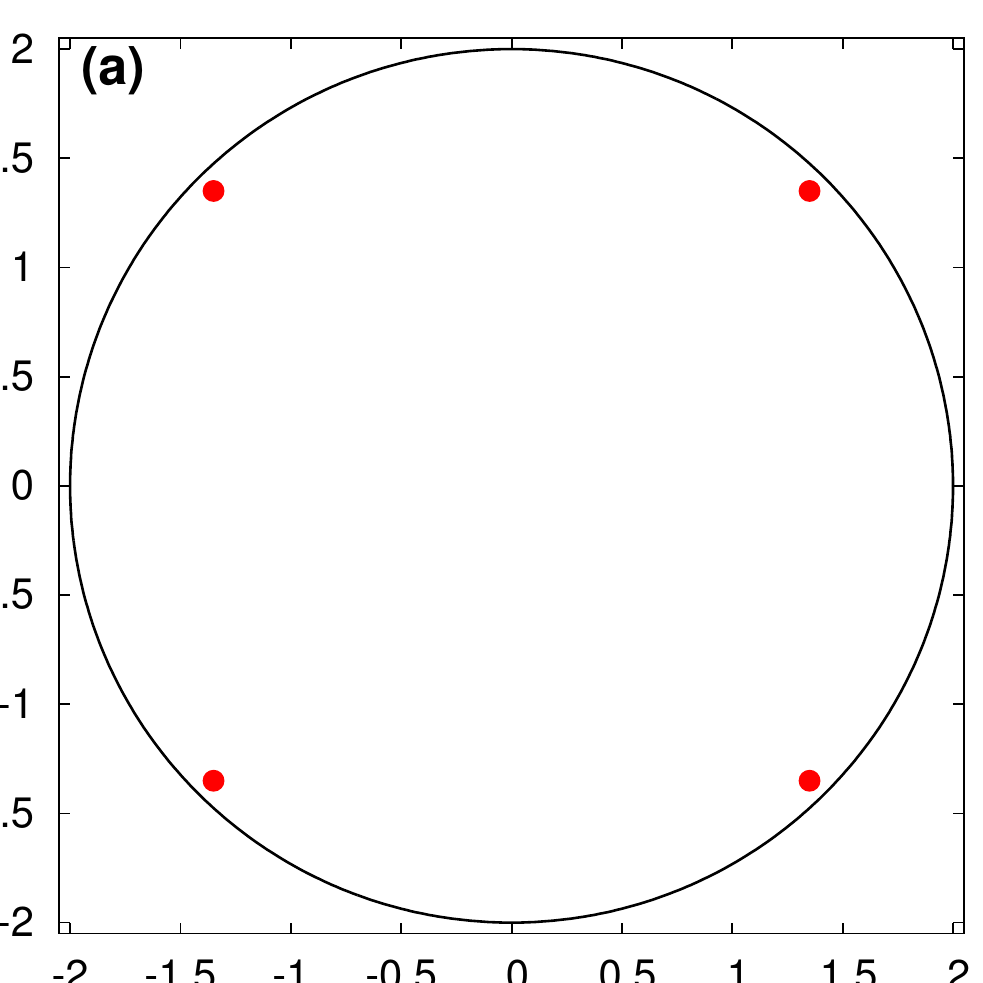}
\includegraphics[width=0.246\textwidth]{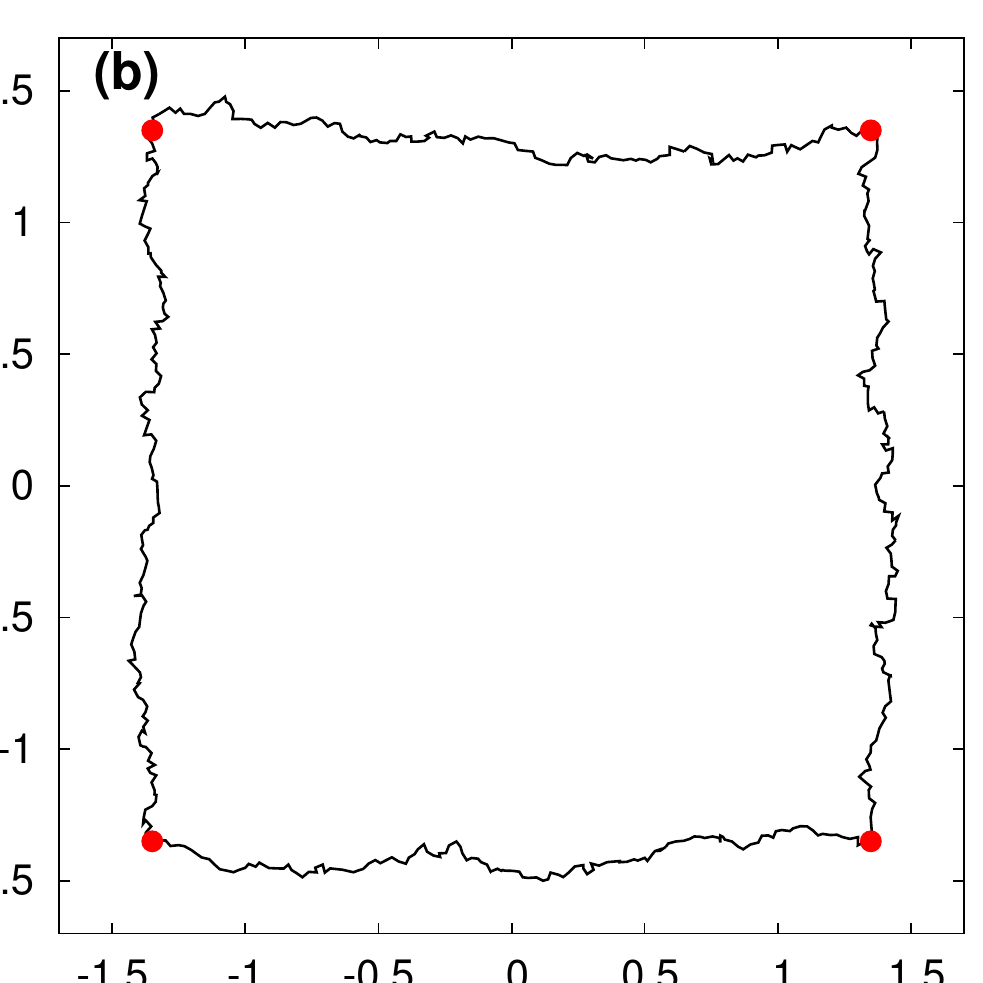}
\includegraphics[width=0.246\textwidth]{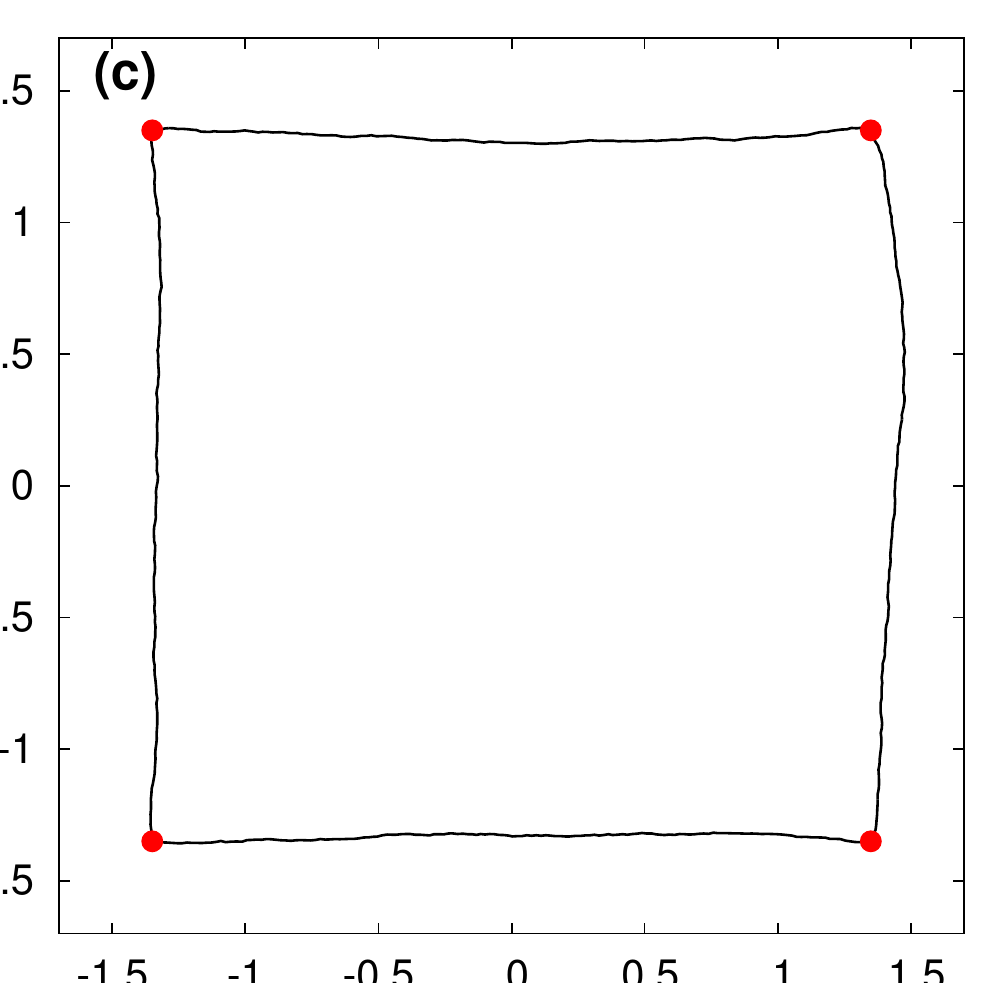}
\includegraphics[width=0.246\textwidth]{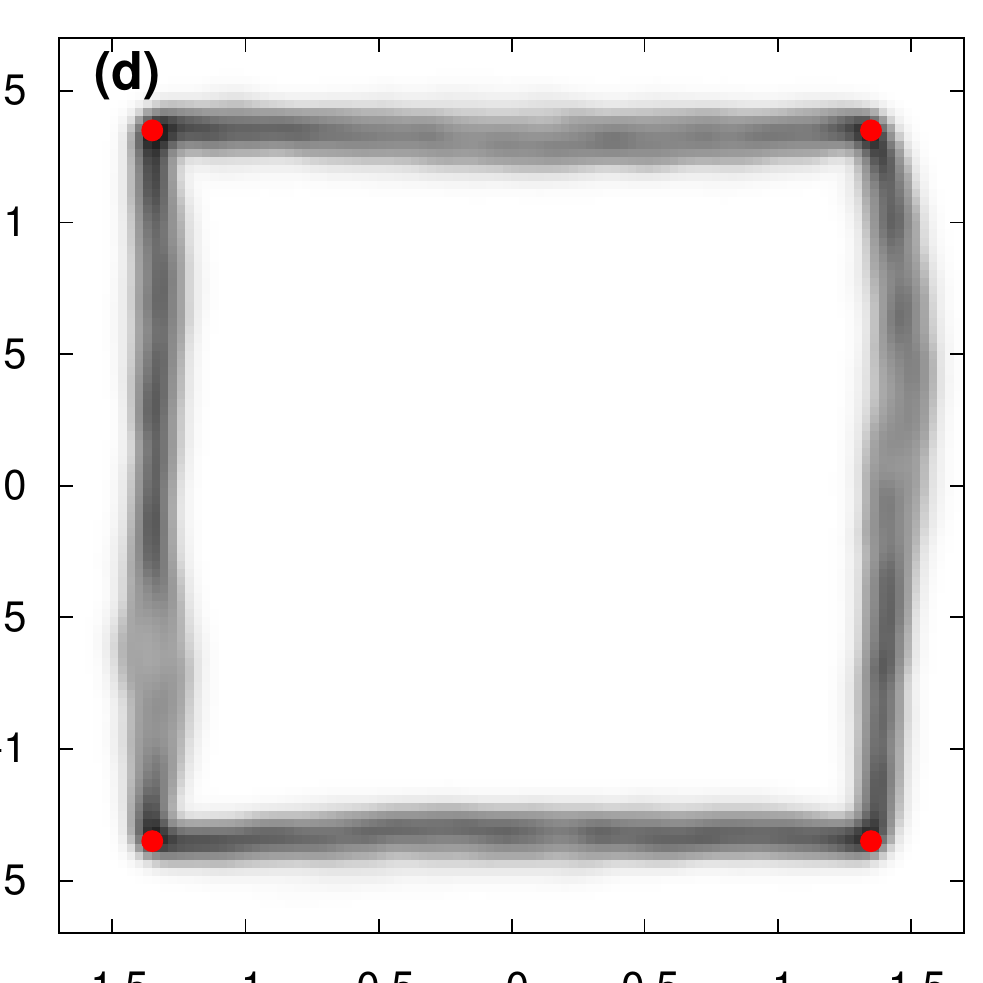}
\caption{\label{fig:sim_sq}{\small
        Example of MCMC simulation: (a) initial loop; (b) loop after
        $2\cdot10^6$ iterations; (c) mean free loop computed using $50$
        representatives; (d) kernel density estimation of vertex
        positions of $50$ loops. Red points indicate punctures.
}}
\end{figure}

\begin{figure}[!htb]
\centering
\includegraphics[width=0.49\textwidth]{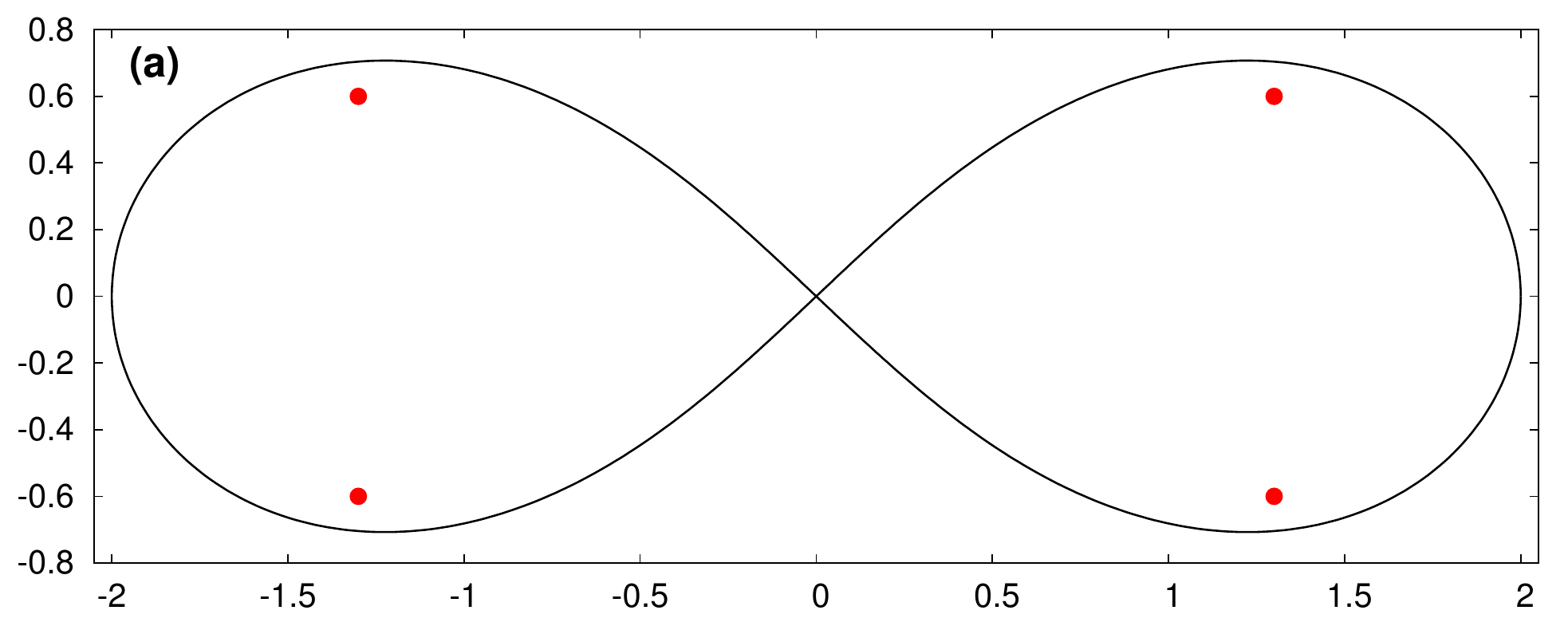}
\includegraphics[width=0.49\textwidth]{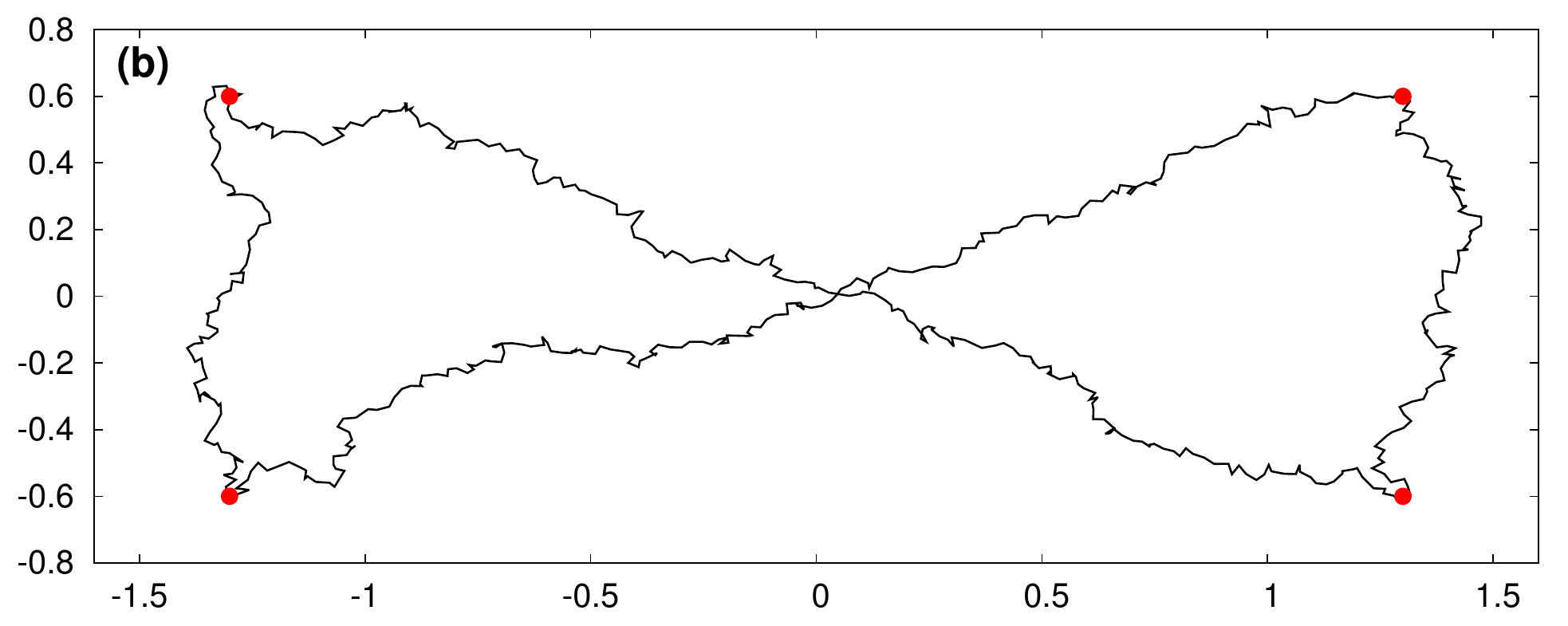}
\includegraphics[width=0.49\textwidth]{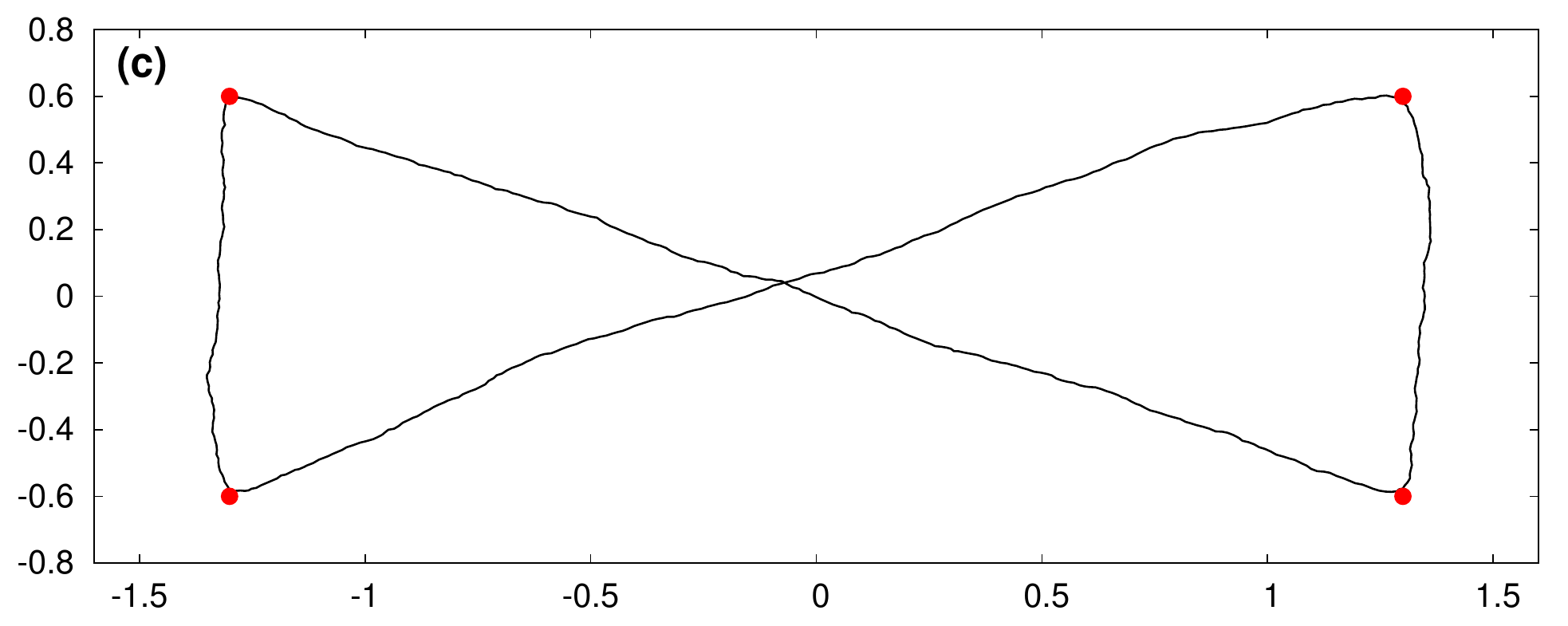}
\includegraphics[width=0.49\textwidth]{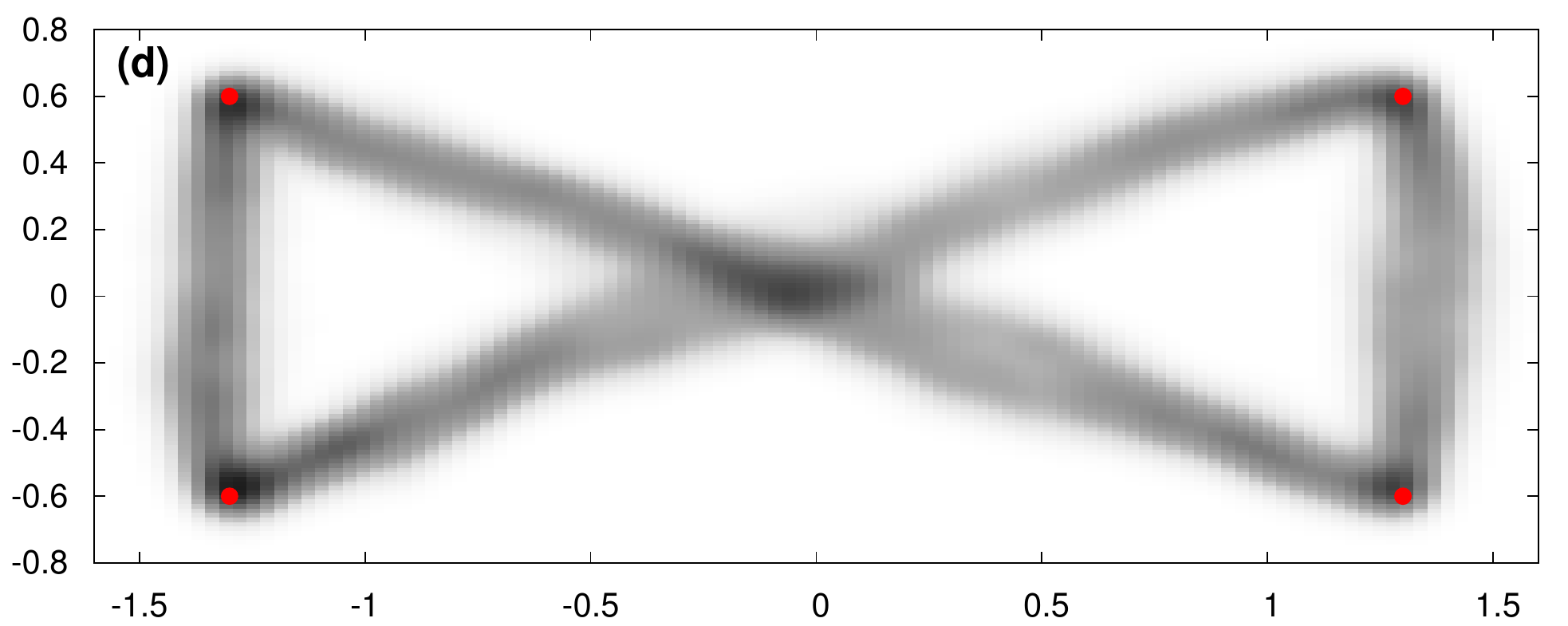}
\caption{\label{fig:sim_bf}{\small
        Another example of MCMC simulation: (a) initial loop; (b) loop after
        $2\cdot10^6$ iterations; (c) mean free loop computed using $50$
        representatives; (d) kernel density estimation of vertex
        positions of $50$ loops. Red points indicate punctures.
}}
\end{figure}
\section{Proofs}
\label{sec:proofs}
We now proceed to prove the results from the previous sections,
starting with preliminary results.

\subsection{Properties of path and loop spaces}

\begin{proof}{(Of Proposition \ref{prp:path_approx}.)}

    Since $\gamma$ is continuous on $[0,1]$ it is uniformly continuous.
    Hence, $\exists$ $\delta>0$ such that
    $\|\gamma(t)-\gamma(s)\|<\frac{\e}{2}$
    whenever $|t-s|<\delta$. Take $m\in\N$ such that
    $\frac{1}{m}<\delta$ and let $t_i=\frac{i}{m}$. Define $\gamma_{PL}$
    to be the piecewise linear
    path with vertices $\gamma\left(t_i\right)$, $i=0,\ldots,m$, and edge
    traversal time $\frac{1}{m}$, so that
    $\gamma_{PL}\left(t_i\right)=\gamma\left(t_i\right)$.
    It is clear that $\gamma_{PL}$ is a loop if $\gamma$ is a loop.
    Also, for $t\in \left[t_{i-1}, t_{i}\right]$, $i=1,\ldots,m$, we have
    $$
    \|\gamma(t)-\gamma_{PL}(t)\|\leq
    \|\gamma(t)-\gamma(t_i)\|+
    \|\gamma_{PL}(t)-\gamma_{PL}(t_i)\|<\e
    $$
\end{proof}

\begin{proof} (Of Proposition \ref{prp:rate_func}.)

    Parts \eqref{i:1}-\eqref{i:3} are standard facts from the large
    deviation theory. Since $\mu$ is invariant under rotations around
    the origin, the same is true for $\Lambda$ and for
    $\Lambda^*$. Hence, $\Lambda(\eta)=\tilde\Lambda(\|\eta\|)$, where
    $\tilde\Lambda$ is a strictly convex, differentiable function on
    $[0,\infty)$. Also, $\Lambda^*(x)=\tilde\Lambda^*(\|x\|)$, where
    $\tilde\Lambda^*$ is a good strictly convex rate function on
    $\D_{\tilde\Lambda^*}$.
    
    To show that $\D_{\Lambda^*}=B_R$ notice that
    $$
    \Lambda(\eta)=\log{\int_{\R^2}{e^{<\eta, x>}\mu(dx)}}=
    R\|\eta\|+\log{\int_{\R^2}{e^{-\|\eta\|(R-x_1)}\mu(dx)}}
    $$
    Dominated convergence theorem yields
    $$
    \lim_{\|\eta\|\to\infty}{\int_{\R^2}{e^{-\|\eta\|(R-x_1)}\mu(dx)}}=
    \int_{\R^2}{\lim_{\|\eta\|\to\infty}e^{-\|\eta\|(R-x_1)}\mu(dx)}=0
    $$
    Therefore for $\|y\|\geq R$ we have
    $$
    \Lambda^*(y) =
    \sup_{\eta}{\left(\|y\|\|\eta\|-\Lambda(\eta)\right)}\geq
    \sup_{\eta}{\left(-\log{\int_{\R^2}{e^{-\|\eta\|(R-x_1)}\mu(dx)}}\right)}=\infty
    $$
    If $\|y\|<R$ then, as we show below, $\exists\eta\in\R^2$ such that
    $y=\nabla\Lambda(\eta)$, and by part \eqref{i:3} we have
    $\Lambda^*(y)=<y,\eta>-\Lambda(\eta)>\infty$.
  
    Now, notice that by the dominated convergence theorem
    $$
    \nabla\Lambda(\eta) =
    e^{-\Lambda(\eta)}\int_{\R^2}{xe^{<x,\eta>}\mu(dx)}
    $$
    Thus, $\nabla\Lambda(0)=0$. If $0<\|y\|<R$ then
    $y=Az$, where $z=(\|y\|, 0)$ and $A$ is a rotation.
    Suppose that $\eta$ is such that $\nabla\Lambda(\eta)=z$. Then
    $$
    \nabla\Lambda(A\eta) =
    e^{-\Lambda(A\eta)}\int_{\R^2}{xe^{<x,A\eta>}\mu(dx)}=
    Ae^{-\Lambda(\eta)}\int_{\R^2}{A^{-1}xe^{<A^{-1}x,\eta>}\mu(dx)}=
    A\nabla\Lambda(\eta)=y
    $$
    Thus, it is enough to show that for each $y=(r,0)$, $0<r<R$,
    we can find $\eta$ such that $\nabla\Lambda(\eta)=y$.

    Take $\eta=(\xi,0)$, then
    $$
    \nabla\Lambda(\eta) = \frac{\int_{\R^2}{xe^{\xi x_1}\mu(dx)}}{\int_{\R^2}{e^{\xi x_1}\mu(dx)}}
    $$
    Notice that $x_2e^{\xi x_1}$ is an odd function of $x_2$, so the
    second coordinate of $\nabla\Lambda(\eta)$ is zero.
    Take $c=R-\e$, $\e>0$. Then
    $$
    \frac{\int_{\R^2}{x_1e^{\xi x_1}\mu(dx)}}{\int_{\R^2}{e^{\xi x_1}\mu(dx)}}=
    \frac{\int_{\R^2}{x_1e^{\xi(x_1-c)}\mu(dx)}}{\int_{\R^2}{e^{\xi(x_1-c)}\mu(dx)}}=
    \frac{\int_{x_1<c}{x_1e^{\xi(x_1-c)}\mu(dx)}+\int_{x_1\geq c}{x_1e^{\xi(x_1-c)}\mu(dx)}}
    {\int_{x_1<c}{e^{\xi(x_1-c)}\mu(dx)}+\int_{x_1\geq c}{e^{\xi(x_1-c)}\mu(dx)}}
    $$
    By the dominated convergence theorem the first term in both
    numerator and denominator goes to zero as $\xi\to\infty$. Also,
    $$
    c\int_{x_1\geq c}{e^{\xi(x_1-c)}\mu(dx)}
    \leq\int_{x_1\geq c}{x_1e^{\xi(x_1-c)}\mu(dx)}\leq
    R\int_{x_1\geq c}{e^{\xi(x_1-c)}\mu(dx)}
    $$
    Hence,
    $$
    \forall\e>0\quad R-\e=c\leq\lim_{\xi\to\infty}\frac{\int_{\R^2}{x_1e^{\xi
    x_1}\mu(dx)}}{\int_{\R^2}{e^{\xi x_1}\mu(dx)}}\leq R
    \implies
    \lim_{\xi\to\infty}\frac{\int_{\R^2}{x_1e^{\xi
    x_1}\mu(dx)}}{\int_{\R^2}{e^{\xi x_1}\mu(dx)}}= R
    $$
    Combining this result with the fact that $\nabla\Lambda(0)=0$ we see
    that there does exist $\xi>0$ such that the first coordinate of
    $\nabla\Lambda(\eta)$ is equal to $r$.
    
\end{proof}

\begin{proof}(Of Theorem \ref{thm:mogulskii2}.)
$I_E$ is a good rate function because $\cl E$ is compact and $I_0$ is a
good rate function. To obtain the lower bound it is enough to show
that for any $\gamma\in\Omega(\R^2)\cap\D_{I_E}$ and $\delta>0$ we have
$$
\liminf_{n\to\infty}{\frac{1}{n}\log{\tilde\mu_n(B_{\delta}(\gamma))}}\geq
-I_E(\gamma),
$$
where
$B_{\delta}(\gamma)=\{\phi\in\Omega(\R^2)|\rho(\gamma, \phi)<\delta\}$.
So, let us take some $\gamma\in\Omega(\R^2)\cap\D_{I_E}$ and $\delta>0$.
For convenience we shall omit the explicit dependence on $\gamma$ from
out notation, so $B_{\delta}=B_{\delta}(\gamma)$.
Let $P_{\delta} = \{\phi(0)|\phi\in B_{\delta}\}$,
$F_{\delta,n}=E_n\cap P_{\delta}$. Notice that
$\tilde\upsilon_n(F_{\delta,n})$ is bounded away from zero for sufficiently large $n$. Given $x\in\R^2$ let
$B^x_{\delta} = \{\phi\in B_{\delta}|\phi(0)=x\}$, and let
$\tilde\mu^x_n$ denote the probability law of the path 
$\tilde\Psi\left(x,\frac{X_1}{n},\ldots,\frac{X_n}{n}\right)$. Then
$$
\tilde\mu_n(B_{\delta})=\int_{F_{\delta,n}}{\tilde\mu^x_n(B^x_{\delta})\tilde\upsilon_n(dx)}
$$

Define $\sigma:\Omega(\R^2)\to\Omega_0(\R^2)$ by
$\sigma(\phi)(t)=\phi(t)-\phi(0)$. Then it is easy to see that
$\tilde\mu^x_n(B^x_{\delta})=\tilde\mu^0_n(\sigma(B^x_{\delta}))$.
Let $D_{\delta} = \cap_{x\in F_{\delta/2,n}}{\sigma(B^x_{\delta})}$.
We claim that $\sigma(B_{\frac{\delta}{2}})\subset D_{\delta}$.
Indeed, if $\phi_0\in \sigma(B_{\frac{\delta}{2}})$ then
$\phi_0(t)=\phi(t)-\phi(0)$ for some $\phi\in B_{\frac{\delta}{2}}$. For
any $x\in F_{\delta/2,n}$ define $\psi_x$ by
$\psi_x(t)=\phi(t)-\phi(0)+x$. Then $\sigma(\psi_x)=\phi_0$ and
$$
\rho(\psi_x,\gamma)=\sup_{t\in [0,1]}{\|\phi(t)-\phi(0)+x-\gamma(t)\|}\leq
\sup_{t\in[0,1]}{\|\phi(t)-\gamma(t)\|}+\|x-\phi(0)\| < \delta,
$$
which proves the claim.
It follows that
$$
\tilde\mu_n(B_{\delta})\geq
\tilde\upsilon_n(F_{\delta/2,n})\tilde\mu^0_n\big(\sigma(B_{\frac{\delta}{2}})\big)
$$
Applying Mogulskii's theorem we get
\begin{align*}
\liminf_{n\to\infty}{\frac{1}{n}\log{\tilde\mu_n(B_{\delta})}}&\geq
\liminf_{n\to\infty}{\frac{1}{n}\left(
    \log{\tilde\upsilon_n(F_{\delta/2,n})}+
    \log{\tilde\mu^0_n(\sigma(B_{\delta/2}))}
\right)}\geq\\
&\geq -\inf_{\phi\in\sigma(B_{\delta/2})}{I_0(\phi)}\geq
-I_0(\sigma(\gamma))=-I_E(\gamma)
\end{align*}

To prove the upper bound suppose that $\Gamma$ is closed. Notice that
$\tilde\mu_n(\Gamma)=\tilde\mu_n(\Gamma\cap\Omega^R(\R^2))$ for all $n$, where
$\Omega^R(\R^2)$ is the set of paths with speed bounded by $R$. Hence, we may assume
that $\Gamma$ consists only of paths with speed bounded by $R$. Then it
follows from the Arzela-Ascoli theorem that $\Gamma$ is compact.
Take $\e>0$. Since $I_0$ is lower semicontinuous, for each $\gamma\in\Gamma$ there exists $\delta_{\gamma}>0$
such that $I(\phi)\geq I(\gamma)-\e$ whenever
$\rho(\phi,\gamma)<4\delta_{\gamma}$. Let $\mathcal{U}$ be a finite
subcover of the cover
$\{B_{\delta_{\gamma}}(\gamma)\}_{\gamma\in\Gamma}$ of $\Gamma$. Denote the
cardinality of $\mathcal{U}$ by $N$. Suppose that
$B_{\delta_{\gamma}(\gamma)}\in\mathcal{U}$. For
convenience we set $\delta=\delta_{\gamma}$ and, once again, omit the
explicit dependence on $\gamma$, so
$B_{\delta}=B_{\delta_{\gamma}}(\gamma)$. Define $F_{\delta,n}$, $B^x_{\delta}$,
$\tilde\mu^x_n$ as before, and notice that $\sigma(B_{\delta})=\cup_{x\in
F_{\delta,n}}{\sigma(B^x_{\delta})}$. Then
$$
\tilde\mu_n(B_{\delta})\leq
\tilde\upsilon_n(F_{\delta,n})\tilde\mu^0_n\big(\sigma(B_{\delta})\big)
$$
Applying Mogulskii's theorem we get
$$
\limsup_{n\to\infty}{\frac{1}{n}\log{\tilde\mu_n(B_{\delta})}}\leq
\limsup_{n\to\infty}{\frac{1}{n}\left(
    \log{\tilde\upsilon_n(F_{\delta,n})}+
    \log{\tilde\mu^0_n(\sigma(B_{\delta}))}
\right)}\leq
-\inf_{\phi\in\sigma(\cl B_{\delta})}{I_0(\phi)},
$$
where $\cl B_{\delta}$ denotes the closure of $B_{\delta}$.
Let $\gamma_0=\sigma(\gamma)$ and notice that for any $\phi_0\in\cl
B_{\delta}$ we have $\phi_0(t)=\phi(t)-\phi(0)$, $\phi\in\cl
B_{\delta}\subset B_{2\delta}$ and
$$
\rho(\gamma_0,\phi_0) =
\sup_{t\in[0,1]}{\|\gamma(t)-\gamma(0)-\phi(t)+\phi(0)\|}\leq
\sup_{t\in[0,1]}{\|\gamma(t)-\phi(t)\|} + \|\phi(0)-\gamma(0)\|
<4\delta
$$
Therefore, $\inf_{\phi\in\sigma(\cl B_{\delta})}{I_0(\phi)}\geq
I_0(\sigma(\gamma))-\e=I_E(\gamma)-\e$. Let
$\tilde\mu_n^*=\max{\{\tilde\mu_n(B_{\delta_{\gamma}}(\gamma))|B_{\delta_{\gamma}}(\gamma)\in\mathcal{U}\}}$
and
$\Gamma^*=\{\gamma\in\Gamma|B_{\delta_{\gamma}}(\gamma)\in\mathcal{U}\}$.
Then
$$
\limsup_{n\to\infty}{\frac{1}{n}\log{\tilde\mu_n^*}}\leq
-\min_{\gamma\in\Gamma^*}I_E(\gamma)+\e,
$$
and so
$$
\limsup_{n\to\infty}{\frac{1}{n}\log{\tilde\mu_n(\Gamma)}}\leq
\limsup_{n\to\infty}{\frac{1}{n}\log{(N\tilde\mu_n^*)}}\leq
\limsup_{n\to\infty}{\frac{1}{n}\log{N}}+
\limsup_{n\to\infty}{\frac{1}{n}\log{\tilde\mu_n^*}}\leq
-\min_{\gamma\in\Gamma}I_E(\gamma)+\e
$$
Since $\e$ is arbitrary, the result follows.
\end{proof}

\begin{proof}(Of Proposition \ref{prp:I_prop}.)

    From the proof of Proposition \ref{prp:rate_func} we have
    $J(\phi)=\int_0^1{\tilde\Lambda^*(\|\phi'(t)\|)dt}$, where
    $\tilde\Lambda^*$ is a good strictly convex rate function on
    $\D_{\tilde\Lambda^*}=[0,R)$. Thus, if $\|\phi'(t)\|\geq
    R$ on a set of positive measure then $I(\phi)=\infty$. This
    proves the first inclusion of \eqref{ii:1}. The second inclusion
    follows from the fact that if $\|\phi'(t)\|<R$ a.e. on $[0,1]$
    then for any $[a,b]\subset[0,1]$, $a<b$, we have
    $L(\phi,a,b)=\int_a^b{\|\phi'(t)\|dt}<R(b-a)$.

    Now, Jensen's inequality implies $J(\phi)\geq
    \tilde\Lambda^*(L(\phi))$, and the equality holds only when $\phi$
    has constant speed, i.e. $\|\phi'(t)\|=L(\phi)$ a.e. on $[0,1]$.
    This proves part \eqref{ii:2}. For part \eqref{ii:3} we then have
    $J(\phi)-J(\gamma)$ $\geq$
    $\tilde\Lambda^*(L(\phi))-\tilde\Lambda^*(L(\gamma))$ $\geq$
    $\tilde\Lambda^*(L(\gamma)+\e)-\tilde\Lambda^*(L(\gamma))$. Let $m$
    be the slope of a supporting line of $\tilde\Lambda^*$ at
    $L(\gamma)$. Notice that $m>0$. Then
    $\tilde\Lambda^*(L(\gamma)+\e)-\tilde\Lambda^*(L(\gamma))$ $\geq$
    $m\e$.
\end{proof}

\begin{proof}(Of Proposition \ref{prp:I_prop}.)

    From the proof of Proposition \ref{prp:rate_func} we have
    $J(\phi)=\int_0^1{\tilde\Lambda^*(\|\phi'(t)\|)dt}$, where
    $\tilde\Lambda^*$ is a good strictly convex rate function on
    $\D_{\tilde\Lambda^*}=[0,R)$. Thus, if $\|\phi'(t)\|\geq
    R$ on a set of positive measure then $I(\phi)=\infty$. This
    proves the first inclusion of \eqref{ii:1}. The second inclusion
    follows from the fact that if $\|\phi'(t)\|<R$ a.e. on $[0,1]$
    then for any $[a,b]\subset[0,1]$, $a<b$, we have
    $L(\phi,a,b)=\int_a^b{\|\phi'(t)\|dt}<R(b-a)$.

    Now, Jensen's inequality implies $J(\phi)\geq
    \tilde\Lambda^*(L(\phi))$, and the equality holds only when $\phi$
    has constant speed, i.e. $\|\phi'(t)\|=L(\phi)$ a.e. on $[0,1]$.
    This proves part \eqref{ii:2}. For part \eqref{ii:3} we then have
    $J(\phi)-J(\gamma)$ $\geq$
    $\tilde\Lambda^*(L(\phi))-\tilde\Lambda^*(L(\gamma))$ $\geq$
    $\tilde\Lambda^*(L(\gamma)+\e)-\tilde\Lambda^*(L(\gamma))$. Let $m$
    be the slope of a supporting line of $\tilde\Lambda^*$ at
    $L(\gamma)$. Notice that $m>0$. Then
    $\tilde\Lambda^*(L(\gamma)+\e)-\tilde\Lambda^*(L(\gamma))$ $\geq$
    $m\e$.
\end{proof}

\begin{proof}(Proof Of Corollary \ref{cor:straight_path}.)

    We shall assume that $\delta$ is small enough so that
    $\Gamma_{\delta}\cap\D_{I_0}\neq\emptyset$, otherwise the result is
    obvious. Notice that this implies that $\delta<2r$.
    By Mogulskii's theorem we have
    $$
    \limsup_{n\to\infty}{\frac{1}{n}\log{\frac{\tilde\mu_n(\Gamma_{\delta})}{\tilde\mu_n(\Gamma)}}}\leq
    \limsup_{n\to\infty}{\frac{1}{n}\log{\tilde\mu_n(\Gamma_{\delta})}}
    - \liminf_{n\to\infty}{\frac{1}{n}{\tilde\mu_n(\Gamma)}}\leq
    -\left(\inf_{\gamma\in\cl\Gamma_{\delta}}{I_0(\gamma)} -
    \inf_{\gamma\in\Gamma^{\circ}}{I_0(\gamma)}\right)
    $$
    Notice that if $\gamma\in\Gamma_{\delta}$ then there exists
    $t\in[0,1]$ such that the shortest distance between $\gamma(t)$ and
    the image of $[0,x^*]$ is at least $\delta$. Then it follows from
    simple geometric considerations that $L(\gamma)\geq
    \sqrt{\|x^*\|^2+\delta^2}$. Since $\delta<2r$ and the square root is
    a concave function we obtain
    $\sqrt{\|x^*\|^2+\delta^2}\geq \|x^*\|+m\delta^2$, where
    $m=\frac{1}{4r^2}(\sqrt{\|x^*\|^2+4r^2}-\|x^*\|)$.
    Proposition \ref{prp:I_prop} then implies that for any
    $\gamma\in\Gamma_{\delta}$ we have $I_0(\gamma)-I_0([0,x^*])\geq
    c\delta^2$, for some constant $c>0$.  
    Also, it is easy to see that
    $\inf_{\gamma\in\Gamma^{\circ}}{I_0(\gamma)}=I_0([0,x^*])$.
    Therefore,
    $$
    -\left(\inf_{\gamma\in\cl\Gamma_{\delta}}{I_0(\gamma)} -
    \inf_{\gamma\in\Gamma^{\circ}}{I_0(\gamma)}\right)\leq -c\delta^2
    $$
\end{proof}

\begin{proof}(Of Lemma \ref{lem:shortest_loop}.)

    Take $\delta\in(0,\reach(Z))$ and consider $X^{\delta}$ with
    the induced length structure and intrinsic metric (see
    \cite{burago2001} for details on length structures). It is easy to
    see that $X^{\delta}$ is a non-positively curved (NPC) space. Hence,
    its universal cover, $\tilde X^{\delta}$, is a Hadamard space
    locally isometric to $X^{\delta}$.

    Let $\ell^{\delta}=\inf_{\gamma\in\homt(X^{\delta})}{L(\gamma)}$,
    $\ell^{*}=\inf_{\gamma\in\homt(X)}{L(\gamma)}$. It follows from
    Cartan's theorem that there is a free loop
    $\hat\gamma^{\delta}\in\hat\homt(X^{\delta})$ such that
    $L(\hat\gamma^{\delta})=\ell^{\delta}$. Moreover, any such free loop
    has a geodesic parametrization
    $\gamma^{\delta}\in\homt(X^{\delta})$. It follows that
    $\hat\gamma^{\delta}$ consists of straight line segments
    which are tangent to (pairs of) circles of radius $\delta$ around the
    punctures and circular arcs connecting such straight line segments
    (see Figure \ref{fig:shortest_loop}).
    \begin{figure}[htb!]
        \includegraphics[width=0.5\textwidth]{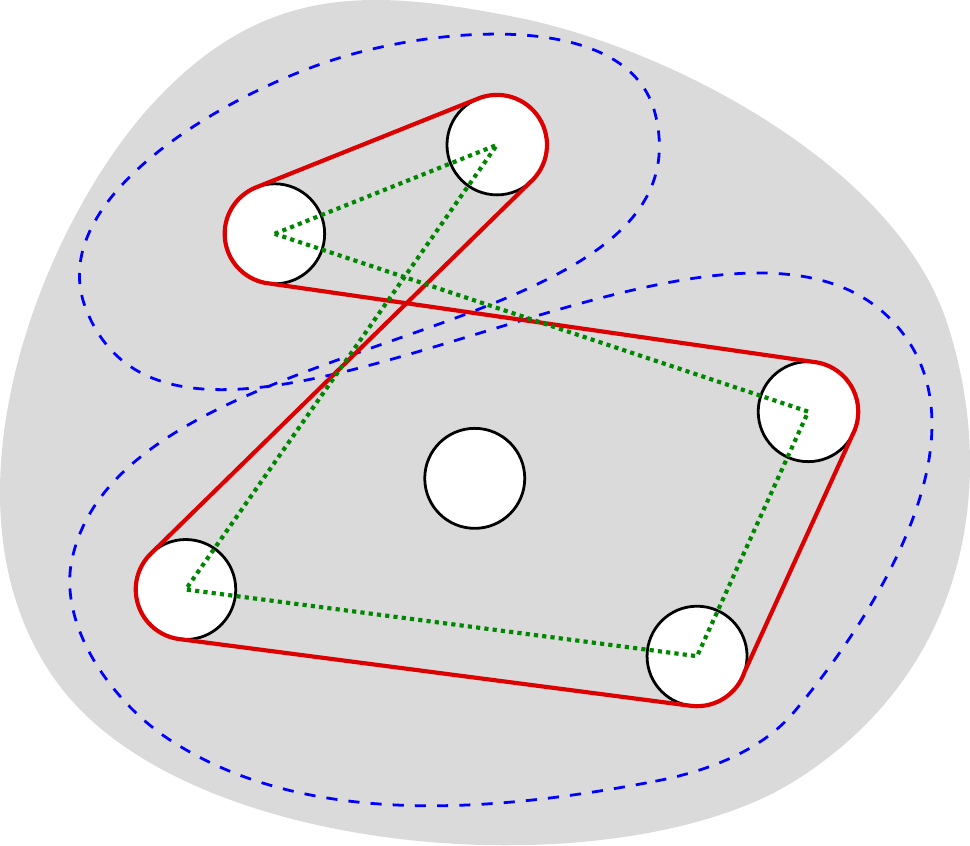}
        \caption{
            \label{fig:shortest_loop}
            The dashed blue line shows a free representative of a free
            homotopy class in the plane with open disks removed. The solid
            red line shows the shortest free loop in the same free homotopy
            class. The dotted green line shows the shortest free
            representative of the corresponding free homotopy class in
            the plane with only centers of the disks removed.
        }
    \end{figure}

    We now show that such a $\hat\gamma^{\delta}$ is unique.  Suppose
    $\hat\gamma_i\in\hat\homt(X^{\delta})$ are such that
    $L(\gamma_i)=\ell^{\delta}$, $i=1,2$. If images of $\hat\gamma_i$
    intersect then we can consider geodesic parametrizations of
    $\hat\gamma_i$ starting at an intersection point. Such closed
    geodesics lift uniquely to geodesics in $\tilde X^{\delta}$
    connecting the same two points. But in a Hadamard space\footnote{Recall
    that a Hadamard space is a complete simply connected space of nonpositive
    curvature.} there is a
    unique geodesic connecting any two points.  Hence,
    $\hat\gamma_1=\hat\gamma_2$, as they have the same geodesic
    representations.
    
    Now assume that $\hat\gamma_1$ and $\hat\gamma_2$ do not intersect.
    A geodesic parametrization of $\hat\gamma_i$, $i=1,2$, is a multiple
    of a simple geodesic, which we denote $\gamma_i$. A periodic
    geodesic defined by $\gamma_i$ can be uniquely lifted to a geodesic
    line $\tilde\gamma_i$ in $\tilde X^{\delta}$, $i=1,2$. Since
    $\hat\gamma_1$ and $\hat\gamma_2$ do not intersect $\tilde\gamma_1$
    and $\tilde\gamma_2$ are parallel. In a Hadamard space parallel
    geodesic lines either coincide or span a convex flat strip. But the
    latter is impossible. Indeed, each $\gamma_i$ does necessarily
    contain a circular arc and $\tilde X^{\delta}$ and $X^{\delta}$ are
    locally isometric, implying that there are points around each
    geodesic line where the metric cannot be flat.

    Let $\delta_m$ be a positive, monotonically decreasing sequence
    converging to zero, and let $\hat\gamma^{\delta_m}$ be the unique
    shortest free loop in $\hat\homt(X^{\delta_m})$. Notice that
    $\lim_{m\to\infty}{L(\hat\gamma^{\delta_m})}=\ell^*$.  Indeed,
    $L(\hat\gamma^{\delta_m})$ is a monotonically increasing sequence
    with a lower bound $\ell^*$, and if a sequence
    $\hat\gamma_i\in\hat\homt(X)$, $i\in\N$ is such that
    $\lim_{i\to\infty}{L(\hat\gamma_i)}=\ell^*$ then for any $i\in\N$ there exists some $M\in\N$ such that for all $m>M$
    $\hat\gamma_i\in\hat\homt(X^{\delta_m})$ $\implies$
    $L(\hat\gamma_i)$ $\geq$ $L(\hat\gamma^{\delta_m})$.  Let
    $\Gamma\subset\homt(X)$ be the set of all constant speed
    parametrizations of all $\hat\gamma^{\delta_m}$, $m\in\N$. Then it
    follows from the Arzela-Ascoli theorem that $\Gamma$ is relatively
    compact (in $\LL(\R^2)$). Hence, we can find a converging (in $\LL(\R^2)$)
    subsequence $\gamma_{m_j}$ of constant speed parametrizations of
    $\hat\gamma^{\delta_{m_j}}$, and
    $\lim_{j\to\infty}{\gamma_{m_j}}=\gamma^*\in\cl\homt(X)$.  Let
    $\hat\gamma^*=\pi_{\LL}(\gamma^*)$. Clearly,
    $L(\hat\gamma^*)=\ell^*$. Moreover, the structure of the shortest
    free loop in $X^{\delta}$ implies that $\hat\gamma^*$ consists of
    straight line segments connecting punctures (see Figure
    \ref{fig:shortest_loop}).

\end{proof}

We now prove our main results: Proposition \ref{prp:main} and Theorem
\ref{thm:main}. The proof of Lemma \ref{lem:lower_bound} is given after
a series of auxiliary technical lemmas following the proof of Theorem
\ref{thm:main}.

\begin{proof}(Of Proposition \ref{prp:main}.)

First, let us prove the upper bound. We may assume that
$\Gamma\cap\D_{I}\neq\emptyset$, otherwise the inequality is trivial.
\begin{align*}
&\limsup_{n\to\infty}{\frac{1}{n}\log{\nu_n(\Gamma)}} = 
\limsup_{n\to\infty}{\left(\frac{1}{n}\log{\left[\mu_n(\iota(\Gamma\cap\homt^R_n(X)))\right]}
- \frac{1}{n}\log{\left[\mu_n(\iota(\homt^R_n(X))\right]}\right)}\leq\\
&\limsup_{n\to\infty}{\frac{1}{n}\log{\left[\mu_n(\iota(\Gamma\cap\homt^R_n(X)))\right]}}
- \liminf_{n\to\infty}{\frac{1}{n}\log{\left[\mu_n(\iota(\homt^R_n(X))\right]}}
\end{align*}
Applying Lemma \ref{lem:lower_bound} to the second term we obtain
$$
\liminf_{n\to\infty}{\frac{1}{n}\log{\left[\mu_n(\iota(\homt^R_n(X))\right]}}
\geq -\inf_{\gamma\in\homt(X)}{I(x)}=-I^*
$$
To bound the first term, take $\e>0$ and let
$$
\Gamma_{\e}=\{\gamma|_{[0,1-\delta]}|\gamma\in\cl\Gamma, 0\leq\delta\leq\e\}
$$
Notice that for sufficiently large $n$ we have
$\Gamma_n\subset\Gamma_{\e}$, where $\Gamma_n =
\iota(\Gamma\cap\homt^R_n(X))$. Therefore,
$$
\limsup_{n\to\infty}{\frac{1}{n}\log{\left[\mu_n(\Gamma_n)\right]}}\leq
\limsup_{n\to\infty}{\frac{1}{n}\log{\left[\mu_n(\Gamma_{\e})\right]}}\leq
-\inf_{\gamma\in\Gamma_{\e}}I(\gamma),
$$
where the last inequality follows from the untethered Mogulskii's theorem.
Take $\gamma\in\cl\Gamma$ and suppose that 
$\gamma_{\e} = \gamma|_{[0,1-\e]}\in\D_I$ for all $\e\geq 0$ (otherwise
$I(\gamma)=I(\gamma_{\e})=\infty$ for sufficiently small $\e$). Then
$$
I(\gamma_{\e})=\int_{0}^{1}{\Lambda^*(\gamma'_{\e}(t))dt}=
\frac{1}{1-\e}\int_{0}^{1-\e}{\Lambda^*((1-\e)\gamma'(s))ds}
$$
Since $\Lambda^*(\cdot)=\tilde\Lambda^*(\|\cdot\|)$ and
$\tilde\Lambda^*$ is a nonnegative increasing function,
the monotone convergence theorem yields $I(\gamma_{\e})\to I(\gamma)$ as
$\e\to 0$. Since $I$ is a good rate function, it attains its infimum on
$\cl\Gamma$ and on $\Gamma_{\e}$. Let $\gamma^*\in\cl\Gamma$ be such
that $I(\gamma^*)=\inf_{\gamma\in\cl\Gamma}{I(\gamma)}$, and let
$I_{\e}=\inf_{\gamma\in\Gamma_{\e}}{I(\gamma)}$. Then
$I_{\e}=I(\gamma_{\e}^*)+\xi(\e)$, where $\gamma^*_{\e} =
\gamma^*|_{[0,1-\e]}$, and $\xi(\e)\to 0$ as $\e\to 0$.
Thus, for all positive $\e$ we have
$$
\limsup_{n\to\infty}{\frac{1}{n}\log{\left[\mu_n(\Gamma_n)\right]}}\leq
-(I(\gamma^*_{\e})+\xi(\e)).
$$
Taking the limit for $\e\to 0$ we get
$$
\limsup_{n\to\infty}{\frac{1}{n}\log{\left[\mu_n(\Gamma_n)\right]}}\leq
-I(\gamma^*)=-\inf_{\gamma\in\cl\Gamma}{I(\gamma)}.
$$

To prove the lower bound, let $\Gamma\subset\homt(X)$ be open. Then
\begin{align*}
&\liminf_{n\to\infty}{\frac{1}{n}\log{\nu_n(\Gamma)}} = 
\liminf_{n\to\infty}{\left(\frac{1}{n}\log{\left[\mu_n(\iota(\Gamma\cap\homt^R_n(X)))\right]}
- \frac{1}{n}\log{\left[\mu_n(\iota(\homt^R_n(X))\right]}\right)}\geq\\
&\liminf_{n\to\infty}{\frac{1}{n}\log{\left[\mu_n(\iota(\Gamma\cap\homt^R_n(X)))\right]}}
- \limsup_{n\to\infty}{\frac{1}{n}\log{\left[\mu_n(\iota(\homt^R_n(X))\right]}}
\end{align*}
Applying Lemma \ref{lem:lower_bound} to the first term we get
$$
\liminf_{n\to\infty}{\frac{1}{n}\log{\left[\mu_n(\iota(\Gamma\cap\homt^R_n(X)))\right]}}\geq
-\inf_{\gamma\in\Gamma}{I(\gamma)}
$$
The second term can be bounded using the same argument as in the case of
the upper bound:
$$
\limsup_{n\to\infty}{\frac{1}{n}\log{\left[\mu_n(\iota(\homt^R_n(X))\right]}}\leq
-\inf_{\gamma\in\cl\homt(X)}{I(\gamma)}=-I^*
$$
\end{proof}

\begin{proof}(Of Theorem \ref{thm:main}.)

    Not surprisingly, the proof is analogous to the proof of Corollary
    \ref{cor:straight_path}.

    We shall assume that $\delta$ is small enough so that
    $\Gamma_{\delta}\cap\D_{I}\neq\emptyset$, otherwise the result is
    obvious. Notice that this implies that $\delta<2R$.
    By Proposition \ref{prp:main}
    $$
    \limsup_{n\to\infty}{\frac{1}{n}\log{\nu_n(\Gamma_{\delta})}}\leq
    -\left(\inf_{\gamma\in\cl\Gamma_{\delta}}{I(\gamma)} - I^*\right)
    $$
    Notice that if $\gamma\in\cl\Gamma_{\delta}$ then there exists
    $t\in[0,1]$ such that the shortest distance between $\gamma(t)$ and
    the image of $\hat\gamma^*$ is at least $\delta$. Then it follows from
    simple geometric considerations that $L(\gamma)\geq
    \sqrt{\ell^2+\delta^2}$, where $\ell$ is the length of
    $\hat\gamma^*$. Since $\delta<2R$ and the square root is
    a concave function we obtain
    $\sqrt{\ell^2+\delta^2}\geq \ell+m\delta^2$, where
    $m=\frac{1}{4R^2}\left(\sqrt{\ell^2+4R^2}-\ell\right)$.
    Proposition \ref{prp:I_prop} then implies that for any
    $\gamma\in\Gamma_{\delta}$ we have $I(\gamma)-I^*\geq
    c\delta^2$, for some constant $c>0$.  
    Therefore,
    $$
    -\left(\inf_{\gamma\in\cl\Gamma_{\delta}}{I(\gamma)} - I^*\right) \leq -c\delta^2
    $$

\end{proof}

The following lemmas, which we needed to prove Lemma \ref{lem:lower_bound},
are adaptations of some standard facts from the large deviation theory.
\begin{lem}
    \label{lem:change_meas}
    Let $\mu_r$ be the uniform probability measure on
    $B_r=\{x\in\R^2|\|x\|<r\}$, $M_r(\eta)$ be the moment generating
    function associated with $\mu_r$, and
    $\Lambda_r(\eta)=\log{M_r(\eta)}$.  Also, let
    $p=\nabla\Lambda_r(\eta)$ for some $\eta\in\R^2$.  Then the random
    variable $Y$ with the probability law $\tilde\mu_r$ defined by
    $$
    \frac{d\tilde\mu_r}{d\mu_r}(x) =
    e^{<x,\eta>-\Lambda_r(\eta)}
    $$
    has expectation $\E(Y)=p$.
\end{lem}
\begin{proof}
    $$
    \E(Y) = \int_{\R^2}{x\tilde\mu_r(dx)} = 
    \int_{\R^2}{xe^{<x,\eta>-\Lambda_r(\eta)}\mu_r(dx)} = 
    \frac{1}{M_r(\eta)}\int_{\R^2}{xe^{<x,\eta>}\mu_r(dx)}
    $$
    On the other hand, $M_r(\eta)=\int_{\R^2}{e^{<x,\eta>}\mu_r(dx)}$ and
    $$
    p=\nabla\Lambda_r(\eta) =\frac{1}{M_r(\eta)}\nabla M(\eta)=
    \frac{1}{M_r(\eta)}\int_{\R^2}{xe^{<x,\eta>}\mu_r(dx)},
    $$
    where the last equality follows form the dominated convergence
    theorem.
\end{proof}

\begin{lem}
    \label{lem:max_ineq}
    Let $X_1,\ldots,X_m$ be i.i.d random variables in $\R^2$
    with $\E(X_1)=0$, and suppose that the values of $X_1$ lie almost
    surely within a set of diameter $c$. Let $S_k=\sum_{i=1}^{k}{X_i}$.
    Then
    $$
    \P\left(\sup_{1\leq k\leq m}{\|S_k\|\geq \lambda}\right)\leq
    4e^{-\frac{\lambda^2}{mc^2}}
    $$
\end{lem}
\begin{proof}
    Let $S_{j,k}$, $j=1,2$, denote the $j$-th coordinate of $S_k$.
    Notice that 
    \begin{align*}
    &\P\left(\sup_{1\leq k\leq m}{\|S_k\|\geq \lambda}\right)\leq
    \P\left(\sup_{1\leq k\leq m}{\max\{|S_{1,k}|, |S_{2,k}|\}}\geq
    \frac{\lambda}{\sqrt{2}}\right)\leq\\
    &\P\left(\sup_{1\leq k\leq m}{|S_{1,k}|}\geq
    \frac{\lambda}{\sqrt{2}}\right) +
    \P\left(\sup_{1\leq k\leq m}{|S_{2,k}|}\geq
    \frac{\lambda}{\sqrt{2}}\right)
    \end{align*}
    Also,
    $$
    \P\left(\sup_{1\leq k\leq m}{|S_{j,k}|}\geq
    \frac{\lambda}{\sqrt{2}}\right)=
    \P\left(\sup_{1\leq k\leq m}{S_{j,k}}\geq
    \frac{\lambda}{\sqrt{2}}\right) +
    \P\left(\sup_{1\leq k\leq m}{(-S_{j,k})}\geq
    \frac{\lambda}{\sqrt{2}}\right),\quad j=1,2
    $$
    Now, for any $t>0$ we have
    $$
    \P\left(\sup_{1\leq k\leq m}{S_{1,k}}\geq
    \frac{\lambda}{\sqrt{2}}\right)=
    \P\left(\sup_{1\leq k\leq m}{e^{tS_{1,k}}}\geq
    e^{\frac{t\lambda}{\sqrt{2}}}\right)
    $$
    Since $X_i$, $i=1,\ldots,m$, have zero expectation, $S_{1,k}$ is a
    martingale. Then it follows from Jensen's inequality that
    $e^{tS_{1,k}}$ is a positive submartingale. Therefore, we can employ
    Doob's martingale inequality to obtain
    $$
    \P\left(\sup_{1\leq k\leq m}{e^{tS_{1,k}}}\geq
    e^{\frac{t\lambda}{\sqrt{2}}}\right)\leq
    \frac{\E(e^{tS_{1,m}})}{e^{\frac{t\lambda}{\sqrt{2}}}}
    $$
    Expanding $S_{1,m}$ and using independence of $X_i$, $i=1,\ldots,m$, we get
    $$
    \E(e^{tS_{1,m}})=\prod_{i=1}^{m}{\E(e^{tX_{1,i}})},
    $$
    where $X_{1,i}$ denotes the first coordinate of $X_i$. Since the
    values of $X_{1,i}$ lie within an interval of lengths $c$,
    Hoeffding's lemma yields
    $$
    \E(e^{tX_{1,i}})\leq e^{\frac{t^2c^2}{8}}
    $$
    Therefore,
    $$
    \P\left(\sup_{1\leq k\leq m}{e^{tS_{1,k}}}\geq
    e^{\frac{t\lambda}{\sqrt{2}}}\right)\leq
    e^{-\frac{t\lambda}{\sqrt{2}}+m\frac{t^2c^2}{8}}
    $$
    Optimizing over $t$ we then obtain
    $$
    \P\left(\sup_{1\leq k\leq m}{e^{tS_{1,k}}}\geq
    e^{\frac{t\lambda}{\sqrt{2}}}\right)\leq
    e^{-\frac{\lambda^2}{mc^2}}
    $$
    The above argument produces the same bound for all four
    probabilities
    $\P\left(\sup_{1\leq k\leq m}{(\pm S_{j,k})}\geq
    \frac{\lambda}{\sqrt{2}}\right)$, $j=1,2$. Thus, we get
    $$
    \P\left(\sup_{1\leq k\leq m}{\|S_k\|\geq \lambda}\right)\leq
    4e^{-\frac{\lambda^2}{mc^2}}
    $$
\end{proof}

Recall that $\mu$ denotes the uniform probability measure on
$B_R=\{x\in\R^2|\|x\|<R\}$, $\Lambda$ denotes the logarithmic moment
generating function associated with the probability law $\mu$, and
$\Lambda^*(x) = \sup_{\eta}{[<x,\eta>-\Lambda(\eta)]}$.

\begin{lem}
    \label{lem:max_path}
    Take $m,n\in\N$, $2\leq m\leq n$, and let
    $X_1,\ldots,X_m$ be i.i.d. random variables with the probability
    law $\mu$. Let $e$ be a linear path in $\R^2$, i.e. $e=[v_0, v_1]$,
    $v_0,v_1\in\R^2$, and let $\gamma$ be a piecewise linear path with edge
    traversal time $\frac{1}{m}$ and vertices
    $v_0+S_k$, $S_k=\frac{1}{n}\sum_{i=1}^{k}{X_i}$, $k=0,\ldots,m$, i.e.
    $\gamma = \tilde\Psi_m
    \left(v_0, \frac{1}{n}X_1,\ldots,\frac{1}{n}X_m\right)$. Suppose
    that $\frac{n}{m}\leq\alpha$, $\alpha\|p\|<R$, $p=v_1-v_0$, and let
    $C>0$. Then there exists a constant $D>0$ such that
    $$
    \P\left(\rho(\gamma,
    e)<\lambda, \|\gamma(1)-e(1)\|<\frac{1}{Cn}\right)\geq
    -\frac{m}{n}\Lambda^*\left(\frac{n}{m}p\right)-\lambda\|\eta_p\|+
    \frac{1}{n}\log{\left(\frac{D}{C^2m}-4e^{-\frac{n^2\lambda^2}{4mR^2}}\right)},
    $$
    where $\eta_p\in\R^2$ is such that
    $\frac{n}{m}p=\nabla\Lambda(\eta_p)$, and $\lambda$ is assumed to be such that
    $\frac{n^2\lambda^2}{4mR^2}>\log{\frac{4C^2m}{D}}$.
\end{lem}
\begin{proof}
    First, notice that existence of $\eta_p$ follows from Proposition
    \ref{prp:rate_func}. Also, since $\gamma(\frac{k}{m})=v_0+S_k$ and
    $e(\frac{k}{m}) = v_0+\frac{kp}{m}$, $k=0,\ldots,m$, we get
    $$
    \P\left(\rho(\gamma, e)<\lambda,
    \|\gamma(1)-e(1)\|<\frac{1}{Cn}\right) = 
    \P\left(\sup_{1\leq k\leq m}{\left\|S_k-\frac{kp}{m}\right\|}<
    \lambda, \left\|S_m-p\right\|<\frac{1}{Cn}\right).
    $$
    Let
    \begin{align*}
    &U_{p,\lambda}=\left\{(x_1,\ldots,x_m)\in\R^{2m}|\sup_{1\leq k\leq
    m}{\left\|\frac{1}{n}\sum_{i=1}^{k}{x_i}-\frac{kp}{m}\right\|}<
    \lambda,
    \left\|\frac{1}{n}\sum_{i=1}^{m}{x_i}-p\right\|<\frac{1}{Cn}\right\},\\
    &U_{0,\lambda}=\left\{(x_1,\ldots,x_m)\in\R^{2m}|\sup_{1\leq k\leq
    m}{\left\|\sum_{i=1}^{k}{x_i}\right\|}< \lambda n,
    \left\|\sum_{i=1}^{m}{x_i}\right\|<\frac{1}{C}
    \right\}
    \end{align*}
    Then
    $$
   \P\left(\sup_{1\leq k\leq m}{\left\|S_k-\frac{kp}{m}\right\|}<
    \lambda, \left\|S_m-p\right\|<\frac{1}{Cn}\right) =
    \int_{U_{p,\lambda}}{\prod_{i=1}^{m}\mu(dx_i)}
    $$
    Letting $\frac{d\tilde\mu}{d\mu}(x) =
    e^{<x,\eta_p>-\Lambda(\eta_p)}$ we get
    \begin{align*}
    \int_{U_{p,\lambda}}{\prod_{i=1}^{m}\mu(dx_i)} &=
    e^{m\Lambda(\eta_p)}\int_{U_{p,\lambda}}{e^{-\sum_{i=1}^{m}{<x_i,\eta_p>}}\prod_{i=1}^{m}\tilde\mu(dx_i)}=\\
    &=e^{m\Lambda(\eta_p)-n<p,\eta_p>}\int_{U_{p,\lambda}}{e^{-\sum_{i=1}^{m}{\left<x_i-\frac{n}{m}p,\eta_p\right>}}\prod_{i=1}^{m}\tilde\mu(dx_i)}=\\
    &=e^{m\Lambda(\eta_p)-n<p,\eta_p>}\int_{U_{0,\lambda}}{e^{-\sum_{i=1}^{m}{<z_i,\eta_p>}}\prod_{i=1}^{m}\bar\mu(dz_i)},
    \end{align*}
    where $\bar\mu$ denotes the probability law of
    $Z_1=Y_1-\frac{n}{m}p$, with $Y_1$ having the probability
    law $\tilde\mu$. Since
    $\left<\sum_{i=1}^{m}{z_i},\eta_p\right>\leq\left\|\sum_{i=1}^{m}{z_i}\right\|\|\eta_p\|$
    and $\left\|\sum_{i=1}^{m}{z_i}\right\|<\lambda n$ on
    $U_{0,\lambda}$, we get
    $$
    \int_{U_{0,\lambda}}{e^{-n\sum_{i=1}^{m}{<z_i,\eta_p>}}\prod_{i=1}^{m}\bar\mu(dz_i)}
    \geq e^{-n\lambda\|\eta_p\|} \P\left(\sup_{1\leq k\leq
    m}{\left\|\sum_{i=1}^k{Z_i}\right\|}< \lambda n,
    \left\|\sum_{i=1}^{m}{Z_i}\right\|<\frac{1}{C}
    \right)
    $$
    Notice that
    $$
    \P\left(\sup_{1\leq k\leq m}{\left\|\sum_{i=1}^k{Z_i}\right\|}<
    \lambda n, \left\|\sum_{i=1}^{m}{Z_i}\right\|<\frac{1}{C} \right)\geq
    \P\left(\left\|\sum_{i=1}^{m}{Z_i}\right\|<\frac{1}{C} \right)
    -
    \P\left(\sup_{1\leq k\leq m}{\left\|\sum_{i=1}^k{Z_i}\right\|}\geq
    \lambda n\right)
    $$
    By Lemma \ref{lem:change_meas} $\E(Y_1)=\frac{n}{m}p$, yielding
    $\E(Z_1)=0$. Moreover, the values of $Z_1$ lie within a disk of
    radius $R$. Hence, we can employ Lemma \ref{lem:max_ineq}
    to obtain
    $$
    \P\left(\sup_{1\leq k\leq m}{\left\|\sum_{i=1}^k{Z_i}\right\|}\geq
    \lambda n\right)\leq 4e^{-\frac{n^2\lambda^2}{4mR^2}}
    $$
    To bound the other probability, notice that the covariance matrix,
    $W$, of $Z_1$ is positive definite, $\E(\|Z_1\|^s)<\infty$
    for all $s\geq 1$, and the density of $Z_1$ is bounded everywhere.
    It follows from the results on uniform local limit theorems
    (see e.g. \cite{bhattacharya1986, petrov1964}) that a bounded
    continuous density, $q_m$, of the distribution of
    $\frac{1}{\sqrt{m}}\sum_{i=1}^{m}{Z_i}$ exists and
    $$
    \left|q_m(x)-\phi_{W}(x)\right|\leq\frac{A}{\sqrt{m}(1+\|x\|^3)},
    \quad\forall x\in\R^2,
    $$
    where $A$ is a constant and $\phi_{W}$ denotes the density of the normal
    distribution in $\R^2$ with zero mean and covariance matrix $W$.
    Denoting by $B_{\frac{1}{C\sqrt{m}}}$ the ball of radius
    $\frac{1}{C\sqrt{m}}$ centered at the origin we then get
    $$
    \P\left(\left\|\sum_{i=1}^{m}{Z_i}\right\|<\frac{1}{C} \right)=
    \P\left(\left\|\frac{1}{\sqrt{m}}\sum_{i=1}^{m}{Z_i}\right\|<\frac{1}{C\sqrt{m}} \right)=
    \int_{B_{\frac{1}{C\sqrt{m}}}}{q_m(x)dx}\leq \frac{D}{C^2m},
    $$
    where $D$ is another constant. Therefore,
    $$
    \frac{1}{n}\log{\P\left(\sup_{1\leq k\leq
    m}{\left\|S_k-\frac{kp}{m}\right\|}< \lambda\right)}\geq
    -\frac{m}{n}\left(\left<\frac{n}{m}p,\eta_p\right>-\Lambda(\eta_p)\right)-\lambda\|\eta_p\|+
    \frac{1}{n}\log{\left(\frac{D}{C^2m}-4e^{-\frac{n^2\lambda^2}{4mR^2}}\right)}
    $$
    The result of the lemma follows from the fact that
    $$
    \left<\frac{n}{m}p,\eta_p\right>-\Lambda(\eta_p)\leq
    \sup_{\eta}{\left(\left<\frac{n}{m}p,\eta\right>-\Lambda(\eta)\right)}=\Lambda^*\left(\frac{n}{m}p\right)
    $$
\end{proof}

\begin{proof}(Of Lemma \ref{lem:lower_bound}.)
It is enough to show that for every $\gamma\in\homt(X)\cap\D_I$
and every $\e>0$ we have
$$
-I(\gamma) \leq
\liminf_{n\to\infty}{\frac{1}{n}\log{\mu_n(\Gamma^{3\e}_n(\gamma))}},
$$
where $\Gamma^{3\e}_n(\gamma)=\iota(\Gamma^{3\e}(\gamma)\cap\homt^R_n(X))$, 
and $\Gamma^{3\e}(\gamma)=\{\alpha\in\homt(X)|\rho(\alpha,\gamma)<3\e\}$
is a ball of radius $3\e$ centered at $\gamma$. Notice that for small
enough $\e$ any loop $\varphi$ such that $\rho(\varphi,\gamma)<3\e$
belongs to $\homt(X)$. Using Proposition
\ref{prp:path_approx} we can find a piecewise linear loop $\gamma_{PL}$
such that $\Gamma^{2\e}(\gamma_{PL})\subset\Gamma^{3\e}(\gamma)$.
Moreover, convexity of $\Lambda^*$ implies that $I(\gamma)\geq
I(\gamma_{PL})$. Therefore, it suffices to show that 
$$
-I(\gamma_{PL}) \leq
\liminf_{n\to\infty}{\frac{1}{n}\log{\mu_n(\Gamma^{\delta}_n(\gamma_{PL}))}},
$$
for $\delta\leq 2\e$.
Denote the vertices of $\gamma_{PL}$ by $v_0,\ldots,v_{\ell}$, and the
edges by $e_0,\ldots, e_{\ell}$. For convenience, we set $v_{\ell+1}=v_0$.
Let $t_i$ be such that
$\gamma_{PL}(t_i)=v_i$, $i=0,\ldots,\ell+1$. As before, denote by $V_0$ the
random variable with the probability law $\upsilon_n$ and by
$X_1,\ldots,X_n$ i.i.d. random variables with the probability law
$\mu$, and let $\psi=\tilde\Psi_n
\left(V_0, \frac{X_1}{n},\ldots,\frac{X_n}{n}\right)$, 
$\varphi=\tilde\Phi_n
\left(V_0, \frac{X_1}{n},\ldots,\frac{X_n}{n}\right)$ (i.e.
$\psi=\iota(\varphi)$).
Then for sufficiently large $n$ we have
$$
\mu_n(\Gamma^{\delta}_n(\gamma_{PL})) = \P(\rho(\varphi,
\gamma_{PL})<\delta)\geq
\P\left(\rho(\psi,\gamma_{PL})<\delta,
\|\psi(0)-\gamma_{PL}(0)\|<\frac{R}{2^{\ell+2}n}, \|\psi(1)-\psi(0)\|<\frac{R}{n}\right)
$$
Denote the right hand side of the above inequality by
$P$. Let $r_i=\frac{R}{2^{\ell+2-i}n}$, and denote by
$B_{r_i}(v_i)$ the open ball of radius $r_i$ centered at $v_i$,
$i=0,\ldots,\ell$.
Given $x\in B_{r_0}(v_0)$ let $\psi_x = \tilde\Psi_n
\left(x, \frac{X_1}{n},\ldots,\frac{X_n}{n}\right)$ and let
$$
P_x = \P\left(\rho(\psi_x,\gamma_{PL})<\delta, \|\psi_x(1)-\psi_x(0)\|<\frac{R}{n}\right)
$$
Notice that $P=\int_{B_{r_0}(v_0)}{P_x \upsilon_n(dx)}$.

We shall now bound $P_x$ from below.
Let $N_i$ be the integer part of $t_i n$, i.e. $N_i=[t_i n]$,
$i=0,\ldots,\ell+1$, and let
$n_i=N_{i+1}-N_{i}$, $i=0,\ldots,\ell$. Take $x_i\in B_{r_i}(v_i)$, and
let $\psi_i=\tilde\Psi_{n_i}
\left(x_i, \frac{X_{n_{i}+1}}{n},\ldots,\frac{X_{n_{i+1}}}{n}\right)$,
$i=0,\ldots,\ell$,
$$
P_i = \P\left(\rho(\psi_i,e_i)<\delta, \|\psi_i(1)-e_i(1)\|<r_{i+1}\right)
$$
Notice that $P_{x_0}\geq \prod_{i=0}^{\ell}{P_i}$ and
$$
P_i\geq \P\left(\rho(\psi_i,e_i)<\delta,
\|\psi_i(1)-(x_i+p_i)\|<r_{i}\right)=Q_i,
$$
where $p_i=v_{i+1}-v_i$.  We also have $t_{i+1}-t_{i}-\frac{1}{n}\leq
\frac{n_i}{n}\leq t_{i+1}-t_{i}+\frac{1}{n}$. Since
$\gamma_{PL}\in\D_I$, the speed of $\gamma_{PL}$ is strictly bounded by
$R$, so $\frac{\|p_i\|}{t_{i+1}-t_{i}}<R$. Hence, for large enough $n$ we
can employ Lemma \ref{lem:max_path} to obtain
$$
\frac{1}{n}\log{P_i}\geq
-\frac{n_i}{n}\Lambda^*\left(\frac{n}{n_i}p_i\right)-\delta\|\eta_{i}\|+\xi_i(n),
$$
where $\eta_i$ are such that $\nabla\Lambda(\eta_i)=\frac{n}{n_i}p_i$,
and $\xi_i(n)\to 0$ as $n\to\infty$. Therefore,
\begin{align*}
&\frac{1}{n}\log{P}=\frac{1}{n}log{\int_{B_{r_0}(v_0)}{P_x
\upsilon_n(dx)}}\geq
\frac{1}{n}\log{\left[\upsilon_n(B_{r_0}(v_0))\right]}+\frac{1}{n}\log{\sum_{i=0}^{\ell}\log{P_i}}\geq
\\
&-\sum_{i=0}^{\ell}{\frac{n_i}{n}\Lambda^*\left(\frac{n}{n_i}p_i\right)}-\delta\sum_{i=0}^{\ell}\|\eta_{i}\|+\xi(n),
\end{align*}
where
$\xi(n)=\sum_{i=0}^{\ell}{\xi_i(n)}+\frac{1}{n}\log{\left[\upsilon_n(B_{r_0}(v_0))\right]}\to
0$ as $n\to\infty$.
Taking the limit we get
$$
\liminf_{n\to\infty}{\frac{1}{n}\log{\mu_n(\Gamma^{\delta}_n(\gamma_{PL}))}}\geq
-\sum_{i=0}^{\ell}{(t_{i+1}-t_{i})\Lambda^*\left(\frac{p_i}{t_{i+1}-t_{i}}\right)}-\delta\sum_{i=0}^{m}\|\eta_{p_i}\|,
$$
where $\eta_{p_i}$ are such that
$\frac{p_i}{t_{i+1}-t_{i}}=\nabla\Lambda(\eta_{p_i})$. Notice that for
$t\in(t_{i}, t_{i+1})$ we have
$\gamma'_{PL}(t)=\frac{p_i}{t_{i+1}-t_{i}}$. Hence,
$$
\liminf_{n\to\infty}{\frac{1}{n}\log{\mu_n(\Gamma^{\delta}_n(\gamma_{PL}))}}\geq
-\sum_{i=0}^{\ell}{\int_{t_i}^{t_{i+1}}{\Lambda^*\left(\gamma'_{PL}(t)\right)}}-\delta S
= -I(\gamma_{PL})-\delta S,
$$
where $S=\sum_{i=0}^{m}\|\eta_{p_i}\|$.
Now,
$$
\liminf_{n\to\infty}{\frac{1}{n}\log{\mu_n(\Gamma^{\delta}_n(\gamma_{PL}))}}\geq
\liminf_{\delta\to
0}\liminf_{n\to\infty}{\frac{1}{n}\log{\mu_n(\Gamma^{\delta}_n(\gamma_{PL}))}}\geq
\liminf_{\delta\to 0}{(-I(\gamma_{PL})-\delta S)} = -I(\gamma_{PL}).
$$

\end{proof}

\subsection{Sampling in $G_n$}
We now turn to the results related to sampling in our loop space $G_n$.

\begin{proof}(Of part \ref{i:samp_main1} of Proposition \ref{prp:sampling_main}.)

    First, we show that $\{\vs_i\}$ is $\nu$-irreducible. Since $G$ is
    an open bounded and connected subset of $R^{2(n+1)}$ and $\nu$ is a
    (rescaled) Lebesgue measure, it is enough to show that each $\vb\in
    G$ has a $\nu$-communicating neighborhood. We call a Borel set
    $B\subset G$ $\nu$-communicating if $\vb\in B$ and all Borel subsets
    $A\subset B$ with $\nu(A)>0$ there exists $m\in\N$ such that
    $P^m(\vb,A)>0$.
    
    Given $\vb=(v_0,\ldots,v_n)\in G$, define $v_{-1}=v_n, v_{n+1}=v_0$,
    and let $O_{\vb}=B_{\delta}(v_0)\times\cdots\times B_{\delta}(v_n)$, where
    $B_{\delta(v_i)}$ denotes a disk of radius $\delta$ centered at
    $v_i$, and $\delta>0$ is such that for all $i=0,\ldots,n$ we have
    $Z\cap B_{\delta}(v_i)=\emptyset$ and $B_{\delta}(v_i)\subset
    B_{\e}(w_{i-1})\cap B_{\e}(w_{i+1})$ for any $w_{i-1}\in
    B_{\delta}(v_{i-1})$, $w_{i+1}\in B_{\delta}(v_{i+1})$.  Define
    $\pi_i:\R^{2(n+1)}\to\R^2$ by $\pi(v_0,\ldots,v_n)=v_i$. It follows
    from the Chapman-Kolmogorov equations that $O_{\vb}$ is
    communicating if for any $\wb\in O_{\vb}$, any $i\in\{0,\ldots,n\}$,
    and any Borel subset $A_i\subset B_{\delta}(v_i)$ with
    $\lambda_2(A_i)>0$, where $\lambda_2$ denotes the $2$-dimensional
    Lebesgue measure, the probability $\P(\pi_i(\vs_{j+1})\in
    A_i|\vs_{j}=\wb)>0$. But it is easy to see that
    this probability is proportional to $\lambda_2(A_i)$.

    To prove aperiodicity it is enough to show that for any Borel set
    $A\subset G$ with $\nu(A)>0$ there exists $\vb\in A$ such that
    $P(\vb,A)>0$. Take $A\subset G$ with $\nu(A)>0$. Define
    $\hat\pi_i:\R^{2(n+1)}\to\R^{2n}$ by
    $\hat\pi_i(v_0,\ldots,v_n)=(v_0,\ldots,v_{i-1},v_{i+1},\ldots,v_n)$.
    Let $\hat{A}_i=\hat\pi_i(A)$, and for $\hat\vb=$
    $(v_0,\ldots,v_{n-1})\in$ $\hat{A}_i$ let 
    $A_i(\hat\vb)=$ $\{v\in\R^2|$
    $(v_0,\ldots,$ $v_{i-1},v,v_{i},\ldots,$ $v_{n-1})\in$ $A\}$. Since $\lambda_{2(n+1)}(A)>0$, where $\lambda_{2(n+1)}$ is the
    $2(n+1)$-dimensional Lebesgue measure, for all $i=0,\ldots,n$ we have
    $\lambda_{2n}(\hat{A}_i)>0$ and there exists $\hat\vb_i\in\hat{A}_i$
    such that $\lambda_2(A_i(\hat\vb_i))>0$. But this implies that
    $P(\vb,A)>0$ for some $\vb\in A$.

    To show that $\{\vs_i\}$ is Harris recurrent it is enough to show
    that for any initial state, with probability $1$, the chain
    eventually moves in every coordinate direction (see Theorem 12 from
    \cite{roberts2006}).  But this is obvious, since the probability
    that a particular vertex does not move after $k$ steps is
    $\left(\frac{n}{n+1}\right)^k\to 0$ as $k\to\infty$.
\end{proof}

The proof of part \ref{i:samp_main2} of Proposition \ref{prp:sampling_main},
which establishes the needed convergence result, relies of several auxiliary
results.

Notice that $G_n$ is path connected if and only if any
$\gamma_0,\gamma_1\in\homt^R_n(X)$ are freely homotopic within
$\homt^R_n(X)$, that is, there exists a free homotopy $H$ between
$\gamma_0$ and $\gamma_1$ such that $H(\cdot,t)\in\homt^R_n(X)$ $\forall
t\in[0,1]$. We denote such a homotopy relation by
$\gamma_0\overset{\homt^R_n}{\simeq}\gamma_1$.

To establish existence of a homotopy within $\homt^R_n(X)$ we employ an
algebraic representation of loops in $\LL(X)$ similar to that in
\cite{grigoriev1998}. Let $\mathcal{T}$ be a collection of arbitrarily
oriented edges in an arbitrary (say, Delaunay) triangulation of the
punctures $Z=\{z_1,\ldots,z_K\}$, including bisectors of the outer
angles of the convex hull of $Z$ (see Figure \ref{fig:triangulation}).
In the degenerate case when all the punctures lie on a single straight
line, let's call it $\ell$, $\mathcal{T}$ consists of the line segments
in $\ell\setminus Z$ and additional rays, two per puncture, which are
perpendicular to $\ell$. Notice that the planar decomposition defined by
$\mathcal{T}$ has convex faces. 

\begin{figure}[htb!]
    \centering
    \includegraphics[width=0.7\textwidth]{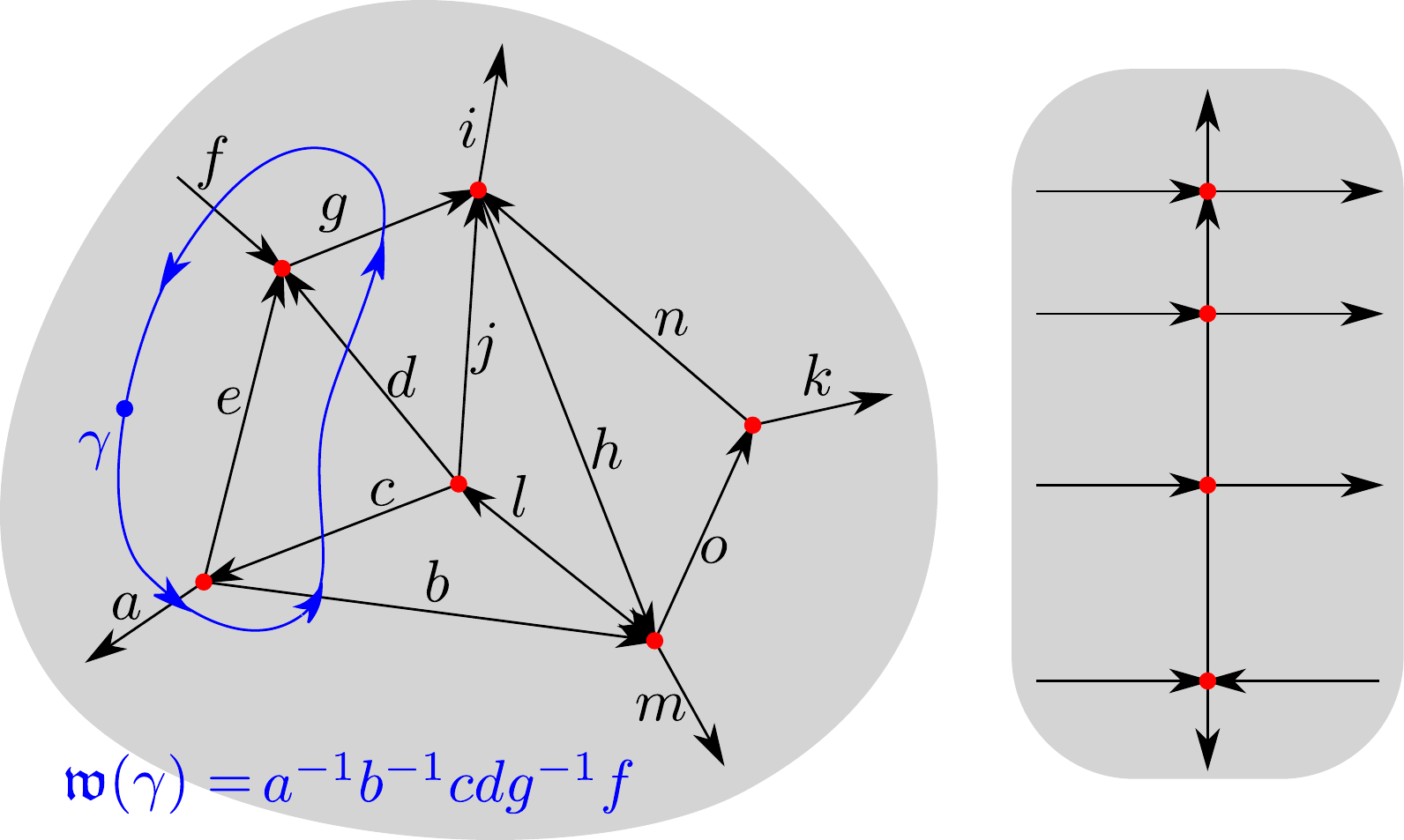}
    \caption{
        \label{fig:triangulation}
        Left: example of an oriented edge collections for a generic
        configuration of punctures, along with associated symbols, a
        loop, and its word representation; Right: degenerate case.
    }
\end{figure}

Associate to each element of $\mathcal{T}$ a symbol, denote the set of
such symbols by $A$, and let $A^{-1}$ be the set of inverse symbols,
i.e. $A^{-1}=\{a^{-1}|a\in A\}$. Let $\mathcal{G}$ be the free group
generated by $A$, and let $\epsilon$ denote the empty word.
Now we can associate to $\gamma\in\LL(X)$
a word over $A$ in the following way. We regard a
loop $\gamma$ as a map from $\R/\Z$ and allow ourselves a slight abuse
notation writing $\gamma(t), t\in\R$, to mean $\gamma(t\,mod\,1)$. Let
$T(\gamma)$ be the collection of connected components of the
intersection of $\gamma$ with $\mathcal{T}$. That is, for each $Q\in
T(\gamma)$ we have $Q\subset E$ for some $E\in\mathcal{T}$, and there
exists a possibly degenerate interval $[s,t]\subset\R$ such that
$\gamma([s,t])=Q$ and $\gamma((s-\e,t+\e))\not\subset E$ $\forall\e>0$.
Notice that $T(\gamma)$ is a finite set. Since
$\gamma([s,t])=\gamma([s+1,t+1])$, we denote by $[t^l_Q,t^r_Q]$ the
first such interval for $Q$ containing non-negative elements.
Generically, each $Q\in T(\gamma)$ is a singleton, but in degenerate
cases some elements of $T(\gamma)$ may be straight line segments. We
order $T(\gamma)$ as follows: for $P,Q\in T(\gamma)$ we define $P\prec
Q\Leftrightarrow t^r_P<t^r_Q$. Suppose $Q\in T(\gamma)$, $Q\subset E$,
$E\in\mathcal{T}$, and let $a\in A$ be the symbol associated to $E$.
Denote by $H_l$ and $H_r$ the left and the right open half spaces
defined by the oriented line corresponding to $E$. We say that $Q$ is a
positive intersection and associate to it the symbol $a$ if
$\exists\e>0$ such that $\gamma((t^l_i-\e,t^l_i))\subset H_l$ and
$\gamma((t^r_i,t^r_i+\e))\subset H_r$. Similarly, $Q$ is a negative
intersection, associated with the symbol $a^{-1}$, if $\exists\e>0$ such
that $\gamma((t^l_i-\e,t^l_i))\subset H_r$ and
$\gamma((t^r_i,t^r_i+\e))\subset H_l$. If $Q$ is neither positive nor
negative, it is said to be a null intersection (and can be associated
with the empty word).  We define $\w(\gamma)$ to be the word
obtained by traversing non-null elements of $T(\gamma)$ in increasing
order and concatenating the corresponding symbols from left to right
(see Figure \ref{fig:triangulation}).

As an element of $\mathcal{G}$, $\w(\gamma)$ may be reduced, i.e.
each pair of consecutive symbols which are inverses of each other is
removed until no such pair exists.  We denote the reduced
$\w(\gamma)$ by $\hat\w(\gamma)$. Notice that $\w(\gamma)$ and
$\hat\w(\gamma)$ represent the same element of $\mathcal{G}$. We call
$\w(\gamma)$ irreducible if $\w(\gamma)=\hat\w(\gamma)$.
Furthermore, $\w(\gamma)$ may be cyclically reduced, meaning that
each pair of cyclically consecutive symbols which are inverses of each
other is removed until no such pair exists. Here, symbols $a$, $b$ in a
word are called cyclically consecutive if they are consecutive or if $a$
is the last symbol and $b$ is the first symbol. A cyclical reduction is
not unique, but any two cyclical reductions of the same word are cyclic
permutations of each other. Let $\W(\gamma)$ denote the set of
all cyclical reductions of $\w(\gamma)$. We call $\w(\gamma)$ cyclically
irreducible if $\w(\gamma)\in\W(\gamma)$. Notice that we can always find
$\omega\in\W(\gamma)$ and $\alpha\in\mathcal{G}$ such
that $\hat\w(\gamma)=\alpha\omega\alpha^{-1}$. Also, since
$\W(\gamma)=\W(\varphi)$ if $\gamma$ and $\varphi$ represent
the same free loop, we define $\W(\hat\gamma)=\W(\gamma)$,
where $\hat\gamma=\pi_{\LL}(\gamma)$, $\gamma\in\LL(X)$.

In what follows, it will be convenient to use some additional notation.
Suppose $\gamma\in\LL(X)$.  For a symbol $a$ in the word
$\w(\gamma)$, let $\kappa_{\gamma}(a)\in T(\gamma)$ be the
intersection associated to $a$. If $a$ and $b$ are two consecutive
symbols in $\w(\gamma)$, we define
$\tau_{\gamma}(a,b)=[t_{\kappa_{\gamma}(a)}^r,
t_{\kappa_{\gamma}(b)}^r]\subset\R$. If $a$ is the last and $b$ is the
first symbol of $\w(\gamma)$, define
$\tau_{\gamma}(a,b)=[t_{\kappa_{\gamma}(a)}^r,
t_{\kappa_{\gamma}(b)}^r+1]\subset\R$. If symbols $a$ and $b$ are not
cyclically consecutive, then there is a sequence of cyclically
consecutive pairs $(a_i, a_{i+1})$, $i=0,\ldots,m$, such that $a_0=a$,
$a_{m+1}=b$, and we define
$\tau_{\gamma}(a,b)=\cup_{i=0}^{m}{\tau_{\gamma}(a_i,a_{i+1})}$.
When it is clear from the context
which loop $\gamma$ is under consideration, we will omit the dependence
on $\gamma$ in our notation and write $\kappa$ and $\tau$.  If
$\gamma\in\LL_{PL}(X)$, then we also define $\tau_{\gamma}^-(a,b)$ and
$\tau_{\gamma}^+(a,b)$ to be the largest (resp.  smallest) closed
interval contained in (resp. containing) $\tau_{\gamma}(a,b)$ whose end
points are vertices of $\gamma$. Finally, for a pair $(a,b)$ of symbols in $\w(\gamma)$ we let
$\gamma|_{a,b}=\gamma|_{\tau(a,b)}$, $L(\gamma,a,b)=L(\gamma|_{a,b})$.

\begin{lem}
    \label{lem:words}
    Loops $\gamma_0,\gamma_1\in\LL(X)$ are homotopic if and only if
    $\hat\w(\gamma_0)=\hat\w(\gamma_1)$.  Furthermore,
    $\gamma_0,\gamma_1$ are freely homotopic if and only if
    $\W(\gamma_0)=\W(\gamma_1)$.
\end{lem}
\begin{proof}
    Notice that $\gamma_0$ and $\gamma_1$ are homotopic if and only if a
    composition $\gamma_0\cdot\bar\gamma_1$ is contractible, where
    $\bar\gamma_1(t)=\gamma_1(1-t)$. Also, $\gamma_0$ and $\gamma_1$ are
    freely homotopic if and only if there exists a path $\varphi$ such
    that $\varphi(0)=\gamma_0(0)$, $\varphi(1)=\gamma_1(0)$, and a
    composition $\gamma_0\cdot\varphi\cdot\bar\gamma_1\cdot\bar\varphi$
    is contractible, where $\bar\varphi(t)=\varphi(1-t)$.  It is clear
    that $\w(\bar\gamma_1)=\w(\gamma_1)^{-1}$, and
    $\w(\varphi\cdot\gamma_1\cdot\bar\varphi)=\sigma\w(\gamma_1)\sigma^{-1}$,
    where $\sigma$ is a word in $\mathcal{G}$. In particular,
    $\W(\varphi\cdot\gamma_1\cdot\bar\varphi)=\W(\gamma_1)$.
    As we show below, a loop $\gamma\in\LL(X)$ is contractible if and
    only if $\W(\gamma)=\{\epsilon\}$. Since
    $\W(\gamma)=\{\epsilon\}\Leftrightarrow
    \hat\w(\gamma)=\epsilon$, it follows that $\gamma_0$ and
    $\gamma_1$ are homotopic if and only if
    $\hat\w(\gamma_0\cdot\bar\gamma_1)=\epsilon$, or equivalently,
    $\hat\w(\gamma_0)=\hat\w(\gamma_1)$. Similarly, $\gamma_0$ is
    freely homotopic to $\gamma_1$ if and only if
    $$
    \hat\w(\gamma_0\cdot\varphi\cdot\bar\gamma_1\cdot\bar\varphi)=\epsilon
    \Leftrightarrow
    \hat\w(\gamma_0)=\hat\w(\varphi\cdot\gamma_1\cdot\bar\varphi).
    $$
    The last equality holds iff there exist
    $\w_{cr}(\gamma_i)\in\W(\gamma_i)$, $i=0,1$,
    and $\alpha,\beta\in\mathcal{G}$ such that
    $$
    \alpha\w_{cr}(\gamma_0)\alpha^{-1}=\beta\w_{cr}(\gamma_1)\beta^{-1}
    \Leftrightarrow
    \w_{cr}(\gamma_0)=\w_{cr}(\gamma_1),
    $$
    which is equivalent to
    $\W(\gamma_0)=\W(\gamma_1)$.

    It remains to show that $\gamma\in\LL(X)$ is contractible if and
    only if $\W(\gamma)=\epsilon$. 
    Notice that if $(a,b)$ is a pair of cyclically consecutive symbols
    of $\w(\gamma)$ which are inverses of each other then
    $\gamma(\tau(a,b))$ belongs to a convex subset of $X$.
    Therefore, we can use a linear homotopy to collapse each
    $\gamma|_{a,b}$ onto the corresponding edge of
    $\mathcal{T}$. It follows that $\gamma$ is freely homotopic to a
    loop $\tilde\gamma$ such that $\w(\gamma)$ is cyclically
    irreducible and $\W(\gamma)=\W(\tilde\gamma)$.

    Thus, if $\W(\gamma)=\{\epsilon\}$ then $\gamma$ is
    freely homotopic to a loop whose image is contained in a convex subset of
    $X$, implying that $\gamma$ is contractible. To prove that a
    contractible $\gamma$ implies $\W(\gamma)=\{\epsilon\}$, we
    suppose that $\W(\gamma)\neq\{\epsilon\}$ and show that
    $\gamma$ cannot be contractible in this case.  We can assume that
    $\w(\gamma)$ is cyclically irreducible. Then for any cyclically consecutive
    symbols $a$ and $b$ of $\w(\gamma)$ $\gamma(\tau(a,b))$ is
    contained in a convex subset of $X$. Thus, we can collapse
    $\gamma|_{a,b}$ onto the straight line segment connecting the
    $\gamma(t_{\kappa(a)}^r)$ and $\gamma(t_{\kappa(b)}^r)$.
    Consequently, $\gamma$ is freely homotopic to a piecewise linear
    loop $\tilde\gamma$ such that $\tilde\gamma(\tau(a,b))$ is a
    straight line segment whenever $a,b$ are cyclically consecutive
    symbols of $\w(\tilde\gamma)$. We can therefore assume that
    $\gamma$ is such a piecewise linear loop. Note that the structure of
    $\mathcal{T}$ implies that $\w(\gamma)$ contains at least three
    symbols. Let $D\subset\R^2\setminus\gamma([0,1])$ be the set of
    points around which $\gamma$ has a non-zero winding number.
    Notice that $D$ is non-empty, open, and bounded, and $\gamma$ cannot
    be contractible if $D$ contains a puncture. Assuming that no
    puncture belongs to $D$ implies that for each symbol $a$ in
    $\w(\gamma)$ there is another symbol $b$ in $\w(\gamma)$
    such that $\gamma(t_{\kappa(a)}^r)$ and
    $\gamma(t_{\kappa(b)}^r)$ belong to the interior of same edge
    from $\mathcal{T}$. It follows that interiors of at least two edges
    from $\mathcal{T}$ intersect, which contradict the definition of
    $\mathcal{T}$.
    
\end{proof}

To prove path connectedness of $G_n$ we employ arguments similar
to those in the proof of Lemma \ref{lem:words}. However, we need to make
sure that $n$ is large enough, so that the corresponding piecewise
linear loop cannot ``get stuck'' around a puncture.

Let $\theta^*$ be the minimum angle in the planar decomposition defined
by $\mathcal{T}$. Notice that if
$n>\frac{2R}{\reach(Z)\sin{\frac{\theta^*}{2}}}$ then an edge of
$\gamma\in\homt^R_n(X)$, say $[v_0,v_1]$, can intersect more than one
edge of $\mathcal{T}$ only if the latter edges are incident to the same
puncture. Moreover, in such a case both $v_0$ and $v_1$ belong to the
ball of radius $\frac{1}{2}\reach(Z)$ centered at this puncture.

Let $\bar\delta=2\reach(Z)\sin{\frac{\theta^*}{2}}$. For
$\delta\in(0,\bar\delta]$, let $\hat\gamma^{\delta}$ denote the shortest
free loop in $\hat\homt(X^{\delta})$, and let $\gamma^{\delta}$ be a
representation of $\hat\gamma^{\delta}$. Recalling the structure of
$\hat\gamma^{\delta}$, we say that a puncture $z\in Z$ is supporting for
$\hat\gamma^{\delta}$ (and for $\gamma^{\delta}$) if the image of
$\hat\gamma^{\delta}$ contains an arc of the circle of radius $\delta$
around $z$. In this case, the circle and the open ball of radius $\delta$ around $z$
will also be called supporting for $\hat\gamma^{\delta}$. We denote the number of supporting punctures for
$\gamma^{\delta}$ by $N^{\delta}$.  Notice that our choice of $\delta$
guarantees that $\w(\gamma^{\delta})$ is cyclically irreducible.

Since $L(\hat\gamma^{\delta})\to L(\hat\gamma^*)$ as $\delta\to 0$,
we define $\delta^*$ $=$ $\sup\left\{
\delta\in(0,\bar\delta]|
R-L(\hat\gamma^{\delta})-\frac{\delta N_{\delta}}{2\reach(Z)}\geq\frac{1}{2}(R-L(\hat\gamma^*))\right\}$,
$N^*=N^{\delta^*}$, and
$n^*=\max{\left\{
        \frac{4R}{\delta^*}+N^*,
        \frac{12N^*R}{R-L(\hat\gamma^*)}
\right\}}$. Our choice of $\delta^*$ implies $n^*>
\frac{2R}{\reach(Z)\sin{\frac{\theta^*}{2}}}$.

\begin{lem}
    \label{lem:simple_pl_hom}
    Let $\gamma\in\homt^R_n(X)$, $n>n^*$, and suppose that $(a,b)$ is a pair
    of cyclically consecutive symbols of $\w(\gamma)$ which are
    inverses of each other. Let
    $\tilde{\w}$ denote the word obtained from $\w(\gamma)$ by
    removing $a$ and $b$. Then there is
    $\tilde\gamma\in\homt^R_n(X)$ such that
    $\gamma\overset{\homt^R_n}{\simeq}\tilde\gamma$ and
    $\w(\tilde\gamma)=\tilde{\w}$.
\end{lem}
\begin{proof}
    For convenience, let $[s,t]=\tau(a,b)$,
    $[s^-,t^-]=\tau^-(a,b)$, $[s^+,t^+]=\tau^+(a,b)$. Also,
    let $e\in\mathcal{T}$ be the edge of the triangulation containing
    $\gamma(s)$ and $\gamma(t)$.

    Notice that $\gamma([s^-,t^-])$ lies in a convex set. Using a linear
    homotopy we can collapse $\gamma_|{[s^-,t^-]}$ onto the straight line
    segment connecting $\gamma(s^-)$ and $\gamma(t^-)$ without
    increasing edge lengths. If
    $\|\gamma(s)-\gamma(t)\|\leq\|\gamma(s^-)-\gamma(t^-)\|$ then we can
    further deform  $\gamma_|{[s^-,t^-]}$ (using a straight line homotopy) to
    make it coincide with the straight line segment connecting
    $\gamma(s)$ and $\gamma(t)$. Thus, we obtain
    $\tilde\gamma\overset{\homt^R_n(X)}{\simeq}\gamma$ such that
    $\tilde\gamma([s^-,t^-])$ is a straight line segment connecting
    $\gamma(s)$ and $\gamma(t)$, and so
    $\w(\tilde\gamma)=\tilde\w$.

    Suppose now that
    $\|\gamma(s)-\gamma(t)\|>\|\gamma(s^-)-\gamma(t^-)\|$.
    Due to the foregoing discussion we can assume that
    $\gamma([s^-,t^-])$ is a straight line segment. Let $v_s$
    and $v_t$ be projections of $\gamma(s^-)$ and $\gamma(t^-)$ onto the
    line through $\gamma(s)$ and $\gamma(t)$. If both $v_s$ and $v_t$
    belong to (the interior of) $e$ then restrictions on $n$ (and hence
    on the edge length) guarantee that triangles with vertices
    $\gamma(s-), \gamma(s+), v_s$ and $\gamma(t-), \gamma(t+), v_t$ do
    not contain punctures. Therefore, we obtain the needed
    $\tilde\gamma$ by linearly homotoping $\gamma_|{[s^-,t^-]}$ onto
    the straight line segment connecting $v_s$ and $v_t$.
    If only one of $v_s$, $v_t$ belongs to (the interior of) $e$, say $v_s$, then 
    $\|v_s-\gamma(t)\|\leq\|\gamma(s^-)-\gamma(t^-)\|$, so the needed
    $\tilde\gamma$ is obtained by linearly homotoping $\gamma_|{[s^-,t^-]}$
    onto the straight line segment connecting $v_s$ and $\gamma(t)$.
    If both $v_s$ and $v_t$ are outside of $e$, then they have to lie on
    the same side of $e$ (otherwise we would have
    $\|\gamma(s)-\gamma(t)\|\leq\|\gamma(s^-)-\gamma(t^-)\|$). This
    implies that $\|\gamma(s)-\gamma(t)\|\leq\sqrt{2}\frac{R}{n}$.
    Moreover, the quadrilateral
    $\gamma(s^+),\gamma(s),\gamma(t),\gamma(t^+)$ does not contain
    punctures. Therefore, the needed $\tilde\gamma$ is obtained by
    homotoping $\gamma_|{[s^-,t^-]}$ onto a line segment of the same (or
    smaller) length centered at the midpoint of the segment connecting
    $\gamma(s)$ and $\gamma(t)$.

\end{proof}

\begin{lem}
    \label{lem:straight_pl_hom}
    Let $\gamma\in\homt^R_n(X)$, $n>n^*$.  Then there is
    $\tilde\gamma\in\homt^R_n(X)$ such that
    $\gamma\overset{\homt^R_n}{\simeq}\tilde\gamma$ and
    $\w(\tilde\gamma)\in\W(\gamma)$. Moreover, for each pair
    $(a,b)$ of cyclically consecutive symbols in $\w(\tilde\gamma)$
    $\tilde\gamma(\tau^-(a,b))$ is a straight line segment.
\end{lem}
\begin{proof}
    Repeatedly applying Lemma \ref{lem:simple_pl_hom} we see that
    $\gamma$ is freely homotopic within $\homt^R_n(X)$ to a loop with
    cyclically irreducible word.  Hence, we may assume that
    $\w(\gamma)\in\W(\gamma)$. Now, let $(a,b)$ be a pair of
    cyclically consecutive symbols in $\w(\gamma)$, and let
    $[s^-,t^-]=\tau^-(a,b)$. Then $\gamma([s^-,t^-])$ is contained in a
    convex set. Therefore, we can linearly homotope
    $\gamma|_{[s^-,t^-]}$ onto the straight line segment connecting
    $\gamma(s^-)$ and $\gamma(t^-)$ without increasing edge lengths.
    Repeating this process for each cyclically consecutive pair of
    symbols yields the needed $\tilde\gamma$.
\end{proof}

We need a few more auxiliary results we can prove connectedness of
$G_n$. We shall say that a (free) homotopy, $H$, is a length
non-increasing homotopy $L(H(\cdot,0),t_1,t_2)\leq
L(H(\cdot,s),t_1,t_2)$ for all $[t_1,t_2]\subset[0,1]$ and $s\in[0,1]$.
Since $X^{\delta}$ is an NPC space, standard results regarding NPC
spaces imply the following (see e.g. Proposition III.1.8 in
\cite{bridson1999}):
\begin{lem}
    Let $\delta\in(0,\reach(Z))$.
    \begin{enumerate}
        \item Suppose $\gamma\in\Omega(X^{\delta})$. Then there exists a
            length non increasing homotopy $H$ of $\gamma$ such that
            $H(\cdot,1)$ is a parametrization of the shortest curve
            between $\gamma(0)$ and $\gamma(1)$ homotopic to $\gamma$.
        \item Suppose $\gamma\in\homt(X^{\delta})$. Then there exists a
            length non increasing free homotopy $H$ of $\gamma$ such that
            $H(\cdot,1)$ is a parametrization of the shortest free loop
            in $\hat\homt(X^{\delta})$.
    \end{enumerate}
\end{lem}

Using the specific structure of our space $X$, we can also prove the
following:
\begin{lem}
    \label{lem:straight_pl_hom2}
    Let $\gamma\in\Omega_{PL}(X)$ be a non self intersecting piecewise
    linear path homotopic to the linear path $\tilde\gamma=[\gamma(0), \gamma(1)]$.
    Then there exists a length non increasing homotopy $H$ of
    $\gamma$ such that $H(\cdot,s)$ is a piecewise
    linear path for each $s\in[0,1]$, $H(t,s)$ is a vertex if and
    only if $\gamma(t)$ is a vertex, and
    $\im(H(\cdot,1))=\im(\tilde\gamma)$.
\end{lem}
\begin{proof}
    For convenience, we shall refer to a homotopy satisfying the conditions
    of the lemma as a proper homotopy.

    First, assume that $\gamma$ and $\tilde\gamma$ intersect only at the
    end points. In this case the loop
    $\varphi=\tilde\gamma\cdot\bar\gamma$, where
    $\bar\gamma(t)=\gamma(1-t)$, defines a simple polygon, $P$.  Let
    $v_0,\ldots,v_m$ be the vertices of $\gamma$ such that
    $v_0=\gamma(0), v_m=\gamma(1)$, and $v_i$, $i=1,\ldots,m-1$ have
    angle different from $\pi$. Let $t_i$ be such that
    $\gamma(t_i)=v_i$.  If $m=2$ then $P$ is a triangle. Hence, $\gamma$
    can be properly homotoped onto (the image of) $\tilde\gamma$ by a
    linear homotopy.  For $m>2$ we can triangulate $P$, with triangles
    having vertices in $\{v_0,\ldots,v_m\}$. A proper homotopy is
    obtained by successively applying a linear homotopy to each part of
    $\gamma$ that passes over two edges of a triangle.

    Suppose now that interiors of $\gamma$ and $\tilde\gamma$ intersect. Let
    $\gamma(s)$ and $\gamma(t)$ be two successive intersection points
    such that $\gamma|_{(s,t)}$ and $\tilde\gamma$ do not intersect. If
    $\gamma(s)$ and $\gamma(t)$ are vertices, we can employ our
    foregoing argument to properly homotope $\gamma_{[s,t]}$ onto
    $[\gamma(s),\gamma(t)]$. Hence, assume that $[s^-,t^-]$ is the
    largest subinterval of $(s,t)$ such that $\gamma(s^-)$ and
    $\gamma(t^-)$ are vertices. Denote by $Q$ the quadrilateral with
    vertices $\gamma(s)$, $\gamma(s^-)$, $\gamma(t^-)$, $\gamma(t)$.
    Let $v_0=\gamma(s^-)$, $v_m=\gamma(t^-)$, and let $v_1,\ldots,v_{m-1}$
    be the vertices of $\gamma|_{(s^-,t^-)}$ lying inside $Q$ such that
    angles $\angle v_{i-1}v_{i}v_{i+1}$, $i=1,\ldots,m-1$ are less than $\pi$.
    Let $t_i$ be such that $v_i=\gamma(t_i)$.
    Then $\gamma|_{[t_i, t_{i+1}]}$ concatenated with $[v_{i+1},v_{i}]$
    defines a simple polygon, and we can use our previous argument to
    properly homotope each $\gamma|_{[t_i,t_{i+1}]}$ onto $[v_i,v_{i+1}]$.
    
    Hence, we can assume that the image of $\gamma|_{[s^-,t^-]}$ is the
    same as the image of a piecewise linear path defined by
    $v_0,\ldots,v_m$. For convenience, let
    $\varphi=\gamma|_{[s^-,t^-]}$, $\varphi_s=\gamma|_{[s,s^-]}$,
    $\varphi_t=\gamma|_{[t^-,t]}$. Suppose that $\varphi$ is not monotone with
    respect to the line $\ell$ defined by $\tilde\gamma$, that is, there
    exists a line perpendicular to $\ell$ intersecting $\varphi$ in more
    than one point. Then we can find a vertex $v_s$ and/or a vertex
    $v_t$ of $\varphi$ such that lines passing through $v_s$ and
    $v_t$, respectively, and perpendicular to the lines defined by
    $\varphi_s$ and $\varphi_t$, respectively, have $\varphi$ on one
    side and do intersect $\varphi_s$ and $\varphi_t$, respectively.
    Let $w_s$ and $w_t$ be the corresponding intersection points, and
    let $s'$ and $t'$ be such that $\varphi(s')=v_s$, $\varphi(t')=v_t$.
    Then we can use a linear homotopy to properly homotope
    $\varphi|_{[0,s']}$ and $\varphi|_{[t',1]}$ onto $[w_s,v_s]$ and
    $[v_t,w_t]$, respectively. Such a deformation makes $\varphi$
    monotone with respect to $\ell$ (see Figure \ref{fig:hom2straight}).
    \begin{figure}[htb!]
        \centering
        \includegraphics[width=0.5\textwidth]{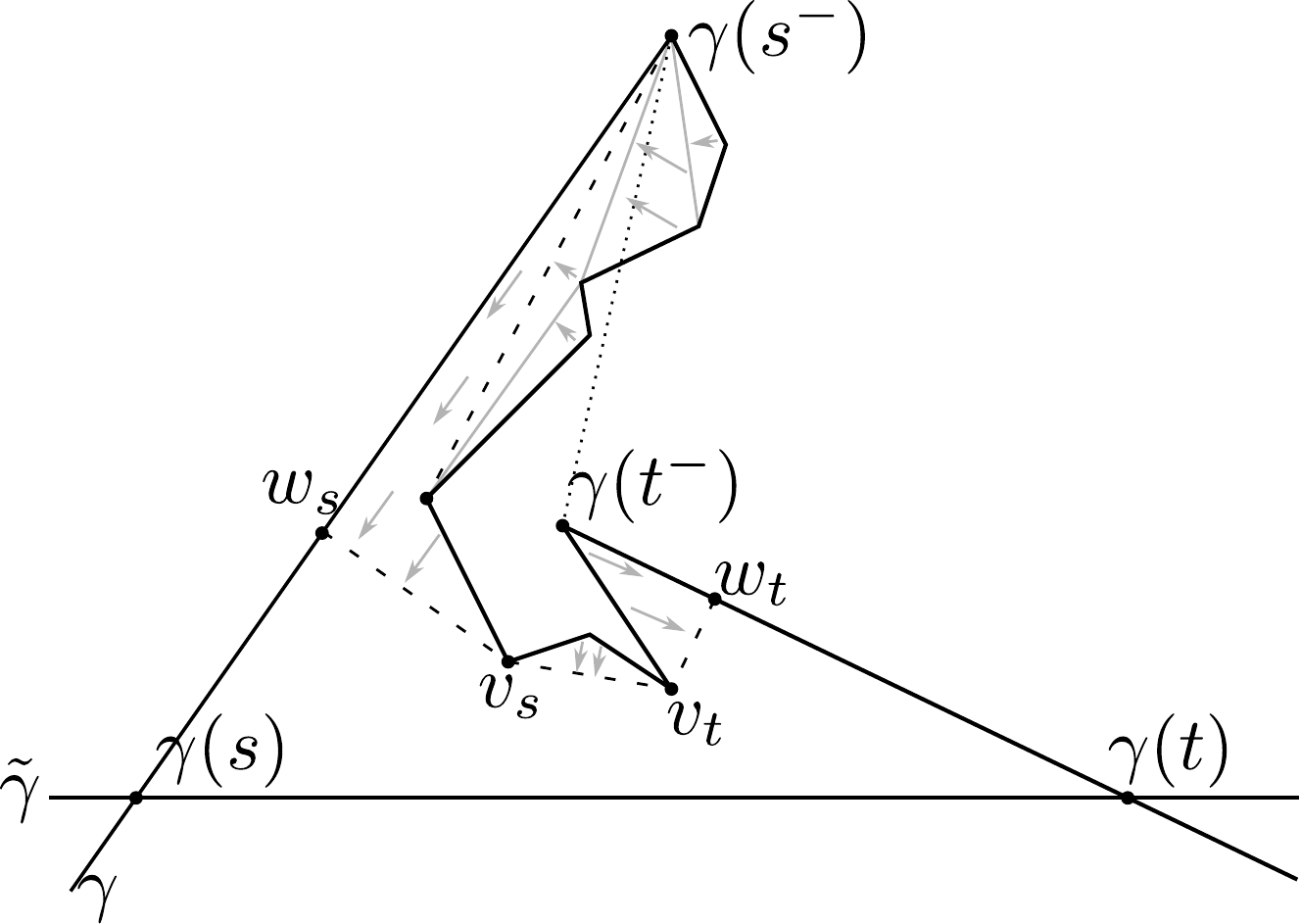}
        \caption{
            \label{fig:hom2straight}
            Illustration of a part of the proof of Lemma
            \ref{lem:straight_pl_hom2}. Gray lines show triangulation of
            the corresponding simple polygon. Gray arrows show direction
            of the homotopy.
        }
    \end{figure}

    The above considerations show that if $\gamma(s)$ and $\gamma(t)$
    are any two successive intersection points of $\gamma$ and
    $\tilde\gamma$ such that $\gamma|_{(s,t)}$ and $\tilde\gamma$ do not
    intersect, then $\gamma|_{[s^-,t^-]}$ can be assumed monotone with
    respect to $\ell$, where $s^-$, $t^-$ and $\ell$ are defined as
    before. But then we can properly homotopy $\gamma$ onto
    $\tilde\gamma$ using a linear homotopy which simply moves the
    vertices of $\gamma$ along the projection lines in such a way that
    all the intersection points stay the same.
\end{proof}

We are now ready to prove that $G_n$ is connected.
\begin{proof}(Of part \ref{i:samp_main2} of Proposition \ref{prp:sampling_main}.)

    We show that if $\gamma_0,\gamma_1\in\homt^R_n(X)$, $n>n^*$, then
    $\gamma_0\overset{\homt^R_n(X)}{\simeq}\gamma_1$.

    Given a loop $\gamma\in\LL(\R^2)$ we shall denote by
    $\sigma_n(\gamma)$ the piecewise linear loop with vertices
    $\gamma(t_i)$, $t_i=\frac{i}{n+1}$, $i=0,\ldots,n$, and edge
    traversal time $\frac{1}{n+1}$.
    
    Let $\gamma^{\delta^*}$ be a constant speed parametrization of
    $\hat\gamma^{\delta^*}$ and let
    $\tilde\gamma^*=\sigma_n(\gamma^{\delta^*})$. Our choice of $n^*$
    guarantees that $\tilde\gamma^*\in\homt^R_n(X)$. We shall show that
    $\tilde\gamma^*\overset{\homt^R_n(X)}{\simeq}\gamma$ for any
    $\gamma\in\homt^R_n(X)$.

    Take $\gamma\in\homt^R_n(X)$. By Lemma \ref{lem:straight_pl_hom}
    we may assume that $\w(\gamma)$ is cyclically irreducible and
    $\gamma(\tau^-(a,b))$ is a straight line segment for each pair
    $(a,b)$ of cyclically consecutive symbols of $\w(\gamma)$.

    First, assume that (the image of) $\gamma$ lies outside of the union
    of open balls of radius $\delta^*$ centered at the punctures.  In
    other words, $\gamma\in\homt(X^{\delta^*})$. Since $X^{\delta^*}$ is
    an NPC space, there exists a length non increasing free homotopy $H$
    of $\gamma$ such that $H(\cdot,1)$ is a parametrization of
    $\hat\gamma^{\delta^*}$. The choice of $n^*$ guarantees that
    $\sigma_n(H(\cdot,s))\in\homt^R_n(X)$ for all $s\in[0,1]$. Let
    $\gamma_1=\sigma_n(H(\cdot,1))$. We say that $\gamma_1$ is
    obtained from $\gamma$ by moving its vertices along $H$.
    The choice of $\delta^*$ allows us
    to further deform $\gamma_1$ by moving its vertices along the image
    of $\gamma^{\delta^*}$ keeping them within $\frac{R}{n+1}$ of each
    other until they coincide with the vertices of $\tilde\gamma^*$.
    Combining such a motion of vertices with $\sigma_n\circ H$ provides
    the homotopy within $\homt^R_n(X)$ between $\gamma$ and
    $\tilde\gamma^*$.

    Suppose now that $\gamma\ni X^{\delta^*}$. Let $[s,t]\subset\R$ be
    such that $\gamma|_{[s,t]}\in\Omega(X^{\delta^*})$, but $\forall\e>0$
    $\gamma|_{[s-\e,t+\e]}\ni\Omega(X^{\delta^*})$. Then there is a
    distance non increasing homotopy $H$ of $\gamma|_{[s,t]}$ such that
    $H(\cdot,1)$ is the shortest path between $\gamma(s)$ and
    $\gamma(t)$ homotopic to $\gamma|_{[s,t]}$. Again, the choice of
    $n^*$ guarantees that moving vertices of $\gamma$ along $H$ is a
    homotopy within $\homt^R_n(X)$. We can perform such a homotopy for each
    of the aforementioned segments $[s,t]$. Hence, we assume that $\gamma$
    has the structure obtained after such deformations.
    
    The loop $\gamma$ may intersect balls which are not supporting for
    $\hat\gamma^{\delta^*}$. Let $[s,t]\subset\R$ be such that
    $\gamma|_{[s,t]}$ lies outside of all supporting balls for
    $\hat\gamma^{\delta^*}$ and $\gamma(s)$, $\gamma(t)$ belong to
    supporting circles. Let $[s^-,t^-]$ be the largest subinterval of
    $[s,t]$ such that $\gamma(s^-)$ and $\gamma(t^-)$ are vertices.
    In this case $\gamma|_{[s^-,t^-]}$ is homotopic to the linear path
    $[\gamma(s^-), \gamma(t^-)]$. Hence, we can employ Lemma
    \ref{lem:straight_pl_hom2} to find a homotopy $H$ of
    $\gamma|_{[s^-,t^-]}$ within $\homt^R_n(X)$ such that $H(\cdot,1)$
    is a re-parametrization of $[\gamma(s^-),\gamma(t^-)]$.
    We can performing such a homotopy for each of the above segments
    $[s,t]$. Hence, we assume that $\gamma$
    has the structure obtained after such deformations.

    We can straighten $\gamma$ a little more. Suppose that $[s,t]$ is
    such that $\gamma|_{[s,t]}$ connects two supporting circles for
    $\hat\gamma^{\delta^*}$. Denote these circles by $C_s$ and $C_t$,
    the corresponding supporting balls by $B_s$, $B_t$, and let $z_s$
    and $z_t$ be the corresponding punctures.  Let $p_s$ and $p_t$ be
    the end points of the corresponding straight line segment of
    $\hat\gamma^{\delta^*}$ (which is tangent to $C_s$ and $C_t$).  Let
    $[s^+,t^+]$ be the largest interval containing $[s,t]$ such that
    $\gamma(s^+)$ and $\gamma(t^+)$ are vertices and
    $\gamma|_{(s^+,t^+)}$ does not intersect $[z_s,p_s]$ and
    $[z_t,p_t]$. Then $\gamma|_{[s^+,t^+]}$ is homotopic to
    $[\gamma(s^+),\gamma(t^+)]$ and we can straighten it using Lemma
    \ref{lem:straight_pl_hom2}. We can perform such straightening for
    each of the segments connecting supporting circles. Hence, we can
    assume that $\gamma$ has the resulting structure. Moreover, since a
    sector of angle less than $\pi$ is convex, the parts of $\gamma$
    within such a sector can also be straightened.  Therefore, we can
    assume that $\gamma$ is such that each $\gamma|_{[s^+,t^+]}$ (with
    $s^+$, $t^+$ as above) is a straight line segment (which we shall
    call a supporting segment of $\gamma$), and the vertices
    of $\gamma$ between supporting segments form a path whose
    length is less than the length of the corresponding circular arc of
    $\hat\gamma^{\delta^*}$.

    The above considerations allow us to assume that $\gamma$ is such
    that
    $$
    R-L(\gamma)\geq
    R-L(\hat\gamma^{\delta^*})-\frac{\delta^*N^*}{2\reach(Z)}-\frac{2N^*R}{n}\geq
    \frac{1}{3}(R-L(\hat\gamma^*))
    $$
    Consequently, we can move the vertices of $\gamma$ along its image,
    keeping them within distance $\frac{1}{n+1}$, until each supporting
    segment of $\gamma$ has the same number of vertices as the part of
    $\tilde\gamma$ lying along the corresponding straight line segment
    of $\hat\gamma^{\delta^*}$, and each part of $\gamma$ between
    supporting segments contains the same number of vertices as the
    corresponding part of $\tilde\gamma^*$. Then we can use a linear
    homotopy to deform $\gamma$ within $\homt^R_n(X)$ onto the image of
    $\tilde\gamma$. If the resulting loop has a different starting point
    than $\tilde\gamma$, we can simply move its vertices along the image
    of $\hat\gamma^{\delta^*}$ to align the starting points.

\end{proof}

\section{Conclusion}
\label{sec:conclusion}
We have extended the Mogulskii's theorem to closed paths in the plane and used
this result to show that the length of a typical representative of a
non-trivial free homotopy class in a multi-punctured plane is extremely close
to the minimum length. We have also provided a simple MCMC method for sampling
from the corresponding uniform measure, thus giving us a way to easily
approximate a solution to the classical problem in geometric optimization.

Of course, using MCMC methods is optimization is not new, but the fact that it
is the uniform measure that is concentrated around the optimum may have
important consequences in several application domains. For example, one may
regard a piecewise linear loop as a closed chain of autonomous agents. Our
result implies that by simply maintaining a proper distance and surrounding
points of interest in a specific way such agents may form a close to optimal
chain, which can be used for relaying signals or other important tasks.

It is not difficult to see that our result should still hold if instead of
punctures we consider any convex obstacles. Moreover, one can expect a similar
result to hold for loops in Riemannian manifolds with a non-trivial fundamental
group. This is one of the directions that we plan to pursue. More generally, it
would be interesting to consider configurations of triangulated surfaces and
other piecewise linear objects, which is likely to require a different
approach.

\bibliographystyle{plain}
\bibliography{refs}
\end{document}